\documentclass[12pt,a4paper]{amsart}

\usepackage[margin=1.25in]{geometry}
\usepackage{csquotes}
\usepackage{amsmath}
\usepackage{amsfonts}
\usepackage{amssymb}
\usepackage{amsthm}
\usepackage{bbm}
\usepackage{mathscinet}
\usepackage{diagbox}
\usepackage{color}
\usepackage{hyperref}
\usepackage{tikz-cd}
\usepackage{rotating}
\usepackage{enumitem}
\usepackage{stackrel}
\usepackage{mathrsfs}
\usepackage[english]{babel}
\definecolor{darkgreen}{rgb}{0,0.4,0.1}
\hypersetup{
   colorlinks=true,       	
   linkcolor=red,          
   citecolor=darkgreen,    
   filecolor=magenta,      
   urlcolor=blue			
}

\numberwithin{equation}{section}
\theoremstyle{plain}
\newtheorem{thm}[equation]{Theorem}
\newtheorem{thmIntro}{Theorem}
\newtheorem{lemma}[equation]{Lemma}

\newtheorem{prop}[equation]{Proposition}
\theoremstyle{definition}
\newtheorem{defn}[equation]{Definition}
\newtheorem{rem}[equation]{Remark}

\newcommand\Z{\mathbb Z}
\newcommand\Q{\mathbb Q}
\newcommand\N{\mathbb N}

\newcommand\Np{{\mathbb N}_{\geq 1}}
\newcommand\Npp{{\mathbb N}_{\geq 2}}
\newcommand\C{\mathbb C}
\newcommand\Lb{\mathbb L}
\newcommand\Lf{\overline{\Lb}}
\newcommand{\im}{\mathsf{i}}
\newcommand{\ts}{{\mathsf t}}

\newcommand{\Rs}{{\mathsf{R}}}
\newcommand{\1}{\mathbbm{1}}

\newcommand{\Vs}{\mathsf{V}}
\newcommand{\Vmd}{\Vs^\alpha}
\newcommand{\Ps}{\mathsf{P}}
\newcommand{\Pmd}{\Ps^\alpha}

\newcommand{\qnum}[1]{\left\{ #1 \right\}}
\newcommand{\qint}[1]{\left[ #1 \right]}
\newcommand{\qbin}[2]{\begin{bmatrix}
#1\\
#2
\end{bmatrix}}

\newcommand{\sltwo}{\mathfrak{sl}(2)}

\newcommand{\AH}{\boldsymbol{A}}
\newcommand{\BH}{\boldsymbol{B}}
\newcommand{\CH}{\boldsymbol{C}}
\newcommand{\NH}{\boldsymbol{N}}
\newcommand{\Vb}{\boldsymbol{V}^\mathrm{JK}}
\newcommand{\RH}{\boldsymbol{R}}
\newcommand{\VH}{\boldsymbol{V}}
\newcommand{\Wb}{\boldsymbol{W}^\mathrm{JK}}
\newcommand{\WH}{\boldsymbol{W}}
\newcommand{\SH}{\boldsymbol{S}}

\newcommand{\HH}{\boldsymbol{H}}

\DeclareMathOperator{\Id}{Id}
\DeclareMathOperator{\Ima}{Im}

\DeclareMathOperator{\Res}{Res}
\DeclareMathOperator{\spn}{span}

\DeclareMathOperator{\End}{End}

\DeclareMathOperator{\soc}{soc}
\DeclareMathOperator{\rad}{rad}
\DeclareMathOperator{\head}{head}

\newcommand{\andq}{\quad \text{ and } \quad}

\newcommand{\forq}{\quad \text{ for } \quad}

\title[Braid group action on projective quantum $\mathfrak{sl}(2)$ modules]{Braid group action on\\[2mm] projective quantum $\mathfrak{sl}(2)$ modules}
\author{Konstantinos Karvounis}
\address{Institut f\"{u}r Mathematik,
Universit\"{a}t Z\"{u}rich,
Winterthurerstrasse 190, CH-8057 Z\"{u}rich, Switzerland.}
\email{konstantinos.karvounis@math.uzh.ch}

\begin{document}

\begin{abstract}
We define a family of the braid group representations via the action of the $R$--matrix  (of the quasitriangular extension) of the restricted quantum $\sltwo$ on a tensor power of a simple projective module. This family is an extension of the  Lawrence representation specialized at roots of unity. Although the center of the braid group has finite order on the
specialized Laurence representations, this action is faithful for our extension.
\end{abstract}

\maketitle

\section*{Introduction}
The braid group $B_n$ on $n$ strands was introduced in 1926 by Artin  
as a group with  generators 
$\sigma_1, ..., \sigma_{n-1}$
and defining relations:
\[
\sigma_i \sigma_j = \sigma_j \sigma_i \; \text{ for } \; \lvert i -j \rvert > 1
\andq
\sigma_i \sigma_{i+1} \sigma_i = \sigma_{i+1} \sigma_i \sigma_{i+1} \; \text{ for } \; 1\leq i \leq n-1.
\]
This group has many  topological incarnations: as a group
of braids in $\mathbb R^3$,  as  a mapping class group of an $n$--punctured disk or as a fundamental group of the configuration space of $n$ points on the plane ect.

In the last decades a long standing problem of the linearity of $B_n$ was resolved by Krammer and Bigelow \cite{Kr1,Bi1, Kr2}. Recall that a group is called linear if it is isomorphic to a subgroup of $GL(m,\mathbb K)$ for some 
 $m\in \mathbb N$ and a field $\mathbb K$. 
They constructed an injective homomorphism from $B_n$
to $GL(n(n-1)/2, \mathbb Z[q^{\pm 1}, t^{\pm 1}])$.
Since the ring of Laurent polynomials in two variables
$q$ and $t$ embeds into real numbers, this implies linearity with $\mathbb K=\mathbb R$.
This representation is  known as the Lawrence--Krammer--Bigelow (LKB) representation, since it fits into the $\ell$--indexed family of 
representations arising from
 the action of $B_n$ on the homology of the configuration space of $\ell$--tuples of $n$--punctured disks constructed by Lawrence \cite{Law} for $\ell=2$ and extending Burau for $\ell=1$.
It remains an interesting open problem  whether linearity can be achieved with rational coefficients.

\vspace{2mm}

Another important recent development is the work of Jackson and Kerler  which relates the Lawrence representations with the theory of  quantum groups. 
In \cite{JK} they construct an $\ell$--indexed family $\Wb_{n,\ell}$ of $B_n$--representations  via the action of the $R$--matrix for the quantum $\sltwo$ on the $n$th tensor power of the generic Verma module $\VH$. Since the $R$--matrix intertwines the action of the quantum group, it preserves the weight space
decomposition of $\VH^{\otimes n}$.
 The representation $\Wb_{n,\ell}$ is then obtained by  restricting  the usual $B_n$--action on $\VH^{\otimes n}$ to the subspace of weight $2 \ell$ less than the highest one.
Jackson and Kerler prove that for $\ell=1$ and $\ell=2$  their representations are isomorphic to the Burau and the LKB representations, respectively, and conjecture that the whole family is isomorphic to the Lawrence representations. Subsequently the conjecture was proved by Ito in \cite{Ito2} as a direct consequence of the Kohno theorem  \cite{Koh}.

In this paper we launch a study of specializations of the Lawrence representations where  both parameters $q$ and $t$ are roots of unity.
This was motivated by   non semisimple TQFTs of
Blanchet--Costantino--Geer--Patureau \cite{BCGP1}, in which the action of the Dehn twist along a separating curve has infinite order.

An appropriate algebraic setting is the representation theory of
the restricted quantum $\sltwo$ at a root of unity $q=e^{\pi \im / r}$ for $r \in \Npp$.
Although, this finite dimensional Hopf algebra 
is not quasitriangular for $r \geq 3$ \cite{KonSa}, there exists a quasitriangular extension $D$  \cite{FGST}. Furthermore, the category of $D$--modules is non semisimple and there is a unique simple projective  module $\Vs_{r-1}$, called Steinberg module whose dimension is $r$.

In this paper we study the action of $B_n$ on the module $\Vs_{r-1}^{\otimes n}$. As a first step we adapt the Jackson--Kerler method to the root of unity case and  get a family of representations $B_n \to \mathrm{GL}(\C)$, denoted by $\WH_{n,\ell}$. As in \cite{JK},
$\WH_{n,\ell}$ is defined as a subspace of weight $2\ell$ less
than highest. In detail, the space $\WH_{n,0}$ is spanned by the highest weight vector of $\Vs_{r-1}^{\otimes n}$ and  hence, is the trivial $B_n$--representation; $\WH_{n,1}$ is  spanned by the $(n-1)$ 
vectors of  weight $2$ less than highest, and so is isomorphic to the specialized Burau representation. For $\ell < r$ the representation $\WH_{n,\ell}$ is a specialization of the corresponding Lawrence representation.

Our main result is a construction of a non trivial
extension of $\WH_{n,\ell}$. For this, we analyze
the structure of $\WH_{n,\ell}$  with respect to the decomposition of $\Vs_{r-1}^{\otimes n}$ into a direct sum of projective modules.
It turns out that if a certain modular condition between the numbers $n, \ell$ and $r$ is satisfied, the representations $\WH_{n,\ell}$ contain  socle vectors of some projective modules. In this case, it is possible to extend the representations $\WH_{n,\ell}$ by including also the dominant head vectors. Thus we obtain the following:

\begin{thmIntro} \label{thm:thm1}
Given  $(n, \ell)\in (\Npp, \N)$, there is
 a pair of natural numbers $(r, \ell')$
 with  $0\leq\ell' < \ell < r$ and
 $\ell' +\ell\equiv 1 -n  \mod r$, such that
 an extension $\NH_{n,\ell,\ell'}$
of 
 $\WH_{n,\ell'}$ by $\WH_{n,\ell}$ is non trivial, i.e.
 the following short exact sequence of $B_n$--modules
\[
\begin{tikzcd}
0 \arrow[r] & \WH_{n,\ell} \arrow[r, hook] & \NH_{n,\ell,\ell'} \arrow[r] &\WH_{n,\ell'} \arrow[r] & 0.
\end{tikzcd}
\]
does not split.
Furthermore,  the center of the braid group acts faithfully on
$\NH_{n,\ell,\ell'}$.
\end{thmIntro}
 
The faithfulness for the center follows from the fact  that the full twist has infinite order on  reducible projective modules. On  $\WH_{n,\ell}$
this action is not faithful. 

An interesting open problem is to extend our construction  to a faithful 
braid group representation over the cyclotomic field.

The Steinberg module of the restricted quantum $\sltwo$ is also a representation of the unrolled quantum $\sltwo$ with integer weights. Therefore, the representations $\NH_{n,\ell,\ell'}$ can also be obtained  by applying the same construction to the Steinberg module of the unrolled quantum $\sltwo$ \cite{KarPhd}, which has been used to construct non semisimple link and 3-manifold invariants in \cite{CGP1, BCGP1}. 
\vspace{2mm}

The special case $\ell=2$ we study in details. Here we obtain two representations:  $\NH_{n,2,0}$ 
and $\NH_{n,2,1}$, 
extending the trivial and specialized Burau representations, respectively, by the specialized LKB $\WH_{n,2}$. 
In both cases, we provide explicit formulas for the braid group action.

Both of them can be generalized to generic three parameter
representations,
$\widetilde{\NH}_{n,2,0}$ and $\widetilde{\NH}_{n,2,1}$ respectively.
However, we prove that $\widetilde{\NH}_{n,2,0}$ splits as a direct sum of the LKB and the trivial representations.
Further we study a certain specialization of $\widetilde{\NH}_{n,2,0}$ at roots of unity, which is isomorphic to $\NH_{n,2,0}$ for $n$ such that the modular condition of Theorem~\ref{thm:thm1} holds. We then prove that this specialization of $\widetilde{\NH}_{n,2,0}$ is a non trivial extension of the trivial by the specialized LKB representation if and only if the modular condition for $n$ is satisfied.

Furthermore, we consider a specialization of $\NH_{n,2,0}$, where we put both parameters equal to $1$. We denote this representation by $\overline{\NH}_{n,2,0}$. The specialized LKB subrepresentation of
$\overline{\NH}_{n,2,0}$  factors through the permutation group $S_n$. Nevertheless, $\NH_{n,2,0}$ is faithful on the subgroup of $B_n$ generated by the half--twist $\Delta_n$, which includes the center of $B_n$.

Finally, we prove that a quotient of the representation $\WH_{n,2}$ is isomorphic to a specialization of the cubic Hecke algebra \cite{MarThese}.

\vspace{2mm}

The paper is organized as follows. In Section~\ref{sec:quantum} we define $D$ and describe the category of $D$-modules. In Section~\ref{sec:WH} we define the highest weight spaces $\WH_{n, \ell}$ and adapt the Jackson-Kerler construction for the root of unity case. Further, we study the representations $\WH_{n,\ell}$ by giving emphasis on their structure with respect to the direct sum decomposition of $\Vs_{r-1}$. Then, in Section~\ref{sec:NH} we define the representations $\NH_{n,\ell,\ell'}$ extending $\WH_{n,\ell'}$ by $\WH_{n,\ell}$ and prove their faithfulness on the center of $B_n$. The representations $\NH_{n,2,0}$ and $\NH_{n,2,1}$ are studied in detail in Section~\ref{sec:examples}, where we give the explicit $B_n$--action. We conclude Section~\ref{sec:examples} by proving that the representations $\NH_{n,2,0}$ and $\NH_{n,2,1}$ are not faithful when $n \geq 3$ and $r \geq 5$. Then in Section~\ref{sec:3var} the representations $\NH_{n,2,0}$ and $\NH_{n,2,1}$ are generalized to $3$--variable representations. We prove that $\widetilde{\NH}_{n,2,0}$ splits and study the restriction of $\widetilde{\NH}_{n,2,0}$ on $B_{n-1}$. Finally, in Section~\ref{sec:cubic} we study how the representation $\WH_{n,2}$ connects to the cubic Hecke algebra.

\subsection*{Acknowledgments} The author would like to thank Prof.~A.~Beliakova and Prof.~C.~Blanchet for the insightful discussions and for providing comments on the current manuscript.

\section{The quasitriagular extension of the restricted quantum \texorpdfstring{$\sltwo$}{sl2}} \label{sec:quantum}

\subsection{Notation}
Let $r \in \Npp$ and $q = e^{\pi \im /r}$ be a $2r$-th primitive root of unity and let $q^{1/2} = e^{\pi \im/2r}$. We set:
\[
\qnum{x} = q^x - q^{-x} \quad \text{and} \quad \qint{x} = \frac{\qnum{x}}{\qnum{1}}.
\]
For $m, n \in \N$ such that $m < r$ we also set
\[
\qnum{n}! = \qnum{n} \qnum{n-1} \cdots \qnum{1}, \quad \qint{n}! = \qint{n} \qint{n-1} \cdots \qint{1}
\]
and for $m,n < r$:
\[
\qbin{n}{m}= \frac{\qint{n}!}{\qint{n-m}!\, \qint{m}!}.
\]

\subsection{Definition of \texorpdfstring{$D$}{D}}
The Hopf algebra $D$ is defined as the $\C$-algebra generated by $E, F, k, k^{-1}$ satisfying the relations
\begin{equation} \label{eq:dq_rels}
\begin{aligned}
E^r = F^r=0, & \quad & k^{4r}=1, & \quad & kk^{-1} = 1
\\
kE = q \, Ek, & \quad & k F = q^{-1} F k, & \quad & \left[E, F \right] = \frac{k^2 - k^{-2}}{q-q^{-1}}.
\end{aligned}
\end{equation}
For the Hopf structure of $D$ we have the following:
\begin{align*} 
&\Delta(E) = 1 \otimes E + E \otimes k^2, & & \epsilon(E)=0, & & S(E)=-E k^{-2},\\
&\Delta(F) = k^{-2} \otimes F + F \otimes 1, & & \epsilon(F)=0, & & S(F)=-k^2 F,\\
&\Delta(k) = k \otimes k, & & \epsilon(k)=1, & & S(k)=k^{-1},\\
&\Delta(k^{-1}) = k^{-1} \otimes k^{-1}, & & \epsilon(k^{-1})=1, & & S(k^{-1})=k.\\
\end{align*}
The restricted quantum $\sltwo$, denoted by $U$, can be defined as the Hopf subalgebra of $D$ generated by $E, F$ and $K := k^2$.
As shown in \cite{FGST}, the Hopf algebra $D$ has a universal $R$-matrix defined by
\begin{equation}\label{eq:Rmatrix}
R = \frac{1}{4r} \sum_{n=0}^{r-1} \sum_{m,m'=0}^{4r-1} \frac{\qnum{1}^{2n}}{\qnum{n}!} q^{n(n-1)/2 + n(m-m') - m m' / 2} E^n k^m \otimes F^n k^{m'},
\end{equation}
as well as a ribbon element defined by
\begin{equation} \label{eq:twist}
\theta = \frac{1 - \im}{2 \sqrt{r}} \sum_{n=0}^{r-1} \sum_{m=0}^{2r-1} \frac{\qnum{1}^{2n}}{\qnum{n}!} q^{-n/2 + nm + (m+r+1)^2/2} F^n E^n k^{2m},
\end{equation}
where $\im = \sqrt{-1}$. Therefore, $(D, R, \theta)$ is a ribbon Hopf algebra \cite{FGST}. Note also that $\theta \in U$.

\subsection{The category of \texorpdfstring{$D$}{D}-modules}
The simple and the indecomposable projective $D$-modules have been classified in \cite{X2,X3}. Note that, any $D$-module $V$ is a \textit{weight module}, that is, $k$ acts diagonally and there exists a direct sum decomposotition into $k$-eigenspaces $V = \oplus_{\lambda \in \C} V_{\lambda}$, where $k v = \lambda v$, for all $v \in V_\lambda$ \cite{X2}. The vector $v$ is called a \textit{weight vector} and the scalar $\lambda$ is called the \textit{weight} of $v$.

Let $\alpha \in \{ \pm 1, \pm \im \}$. Recall that $q^{1/2} = e^{\pi \im/2r}$. In $D$-mod there exist $r$ simple modules $\Vmd_i$ of highest weight $\alpha q^{i/2}$, where $0 \leq i \leq r-1$ and with $\dim \Vmd_i = i + 1$. They are generated by the basis (weight) vectors $u_0^\alpha, \ldots, u_i^\alpha$ satisfying the relations
\begin{equation}\label{eq:action_Va}
k \, u_m^\alpha = \alpha q^{\frac{i-2m}{2}} \, u_m^\alpha, \quad
E \, u_m^\alpha = u_{m-1}^\alpha \andq
F \, u_m^\alpha = \alpha^2 \qint{m + 1} \qint{i-m} \, u_{m+1}^\alpha,
\end{equation}
where $u_{-1} = u_{i+1} = 0$. The modules $\Vs_{r-1}^\alpha$ are projective \cite[Lemma 2.1.11]{X3}. In this manuscript, we are concerned mainly with the module $\Vs_{r-1}^{1}$. To simplify the notation we set $\Vs_i := \Vs_i^1$ and similarly $u_i := u_i^1$. Additionally, we set
\begin{equation}\label{eq:s}
s := q^{r - 1}.
\end{equation}
Then, the relations satisfied by the basis vectors of $\Vs_{r-1}$ are written as
\[
k \, u_m = s^{1/2} q^{-m} \, u_m, \quad
E \, u_m = u_{m-1} \andq
F \, u_m = \qint{m + 1} \qint{r - 1 -m} \, u_{m+1},
\]
where $0 \leq m \leq r-1$ and $u_{-1} = u_r = 0$.
Note also that $\Vmd_0 \otimes \Vs_i \cong \Vmd_i$.

\begin{rem} \label{rem:V_basis_F}
In \cite{X2} another basis $\{v_0^\alpha, \ldots, v_{r-1}^\alpha \}$ of $\Vmd_i$ is used. The change of basis is given by $v_m^\alpha = (-1)^i \qint{m}! \qint{i+1-m}! \, u_i^\alpha$. The $D$-action on $\Vmd_i$ is given by
\[
k \, v_m^\alpha = \alpha q^{\frac{i-2m}{2}} \, v_m^\alpha, \quad
E \, v_m^\alpha = \alpha^2 \qint{m} \qint{i+1-m} \, v_{m-1}^\alpha \andq
F \, v_m^\alpha = v_{m+1}^\alpha.
\]
\end{rem}

\bigbreak 
In $D$-mod, there exist $2r$-dimensional non-simple indecomposable projective modules $\Pmd_i$, where $0 \leq i \leq r-2$ and as before $\alpha \in \{ \pm 1, \pm \im \}$, that are the projective covers of the modules $\Vmd_i$. Figure~\ref{fig:proj_D} illustrates the structure of the modules $\Pmd_i$, where we set $j := r - 2 -i$.

\begin{figure}[ht] 
\begin{alignat*}{3}
&\raisebox{-\height}{$\swarrow$} & \; \Vmd_i \; & \raisebox{-\height}{$\searrow$}& &\; \text{(head)}
\\
\Vs_j^{-\im \alpha} & & & & \Vs_j^{\im \alpha}
\\
&\raisebox{\height}{$\searrow$}  & \; \Vmd_i \; & \raisebox{\height}{$\swarrow$}& &\; \text{(socle)}
\end{alignat*}
\caption{The projective module $\Pmd_i$}\label{fig:proj_D}
\end{figure}

For any module $M$, the \textit{socle} $\soc(M)$ of $M$ is defined as the sum of all irreducible submodules of $M$. Further, the \textit{radical} $\rad(M)$ of $M$ is defined as the intersection of all maximal submodules of $M$. We also define the \textit{head} of $M$ by $\head(M)=M/\rad(M)$. For the non-simple indecomposable projective $\Pmd_i$, it is easy to see that
\[
\soc(\Pmd_i)=\Vmd_i , \qquad \rad(\Pmd_i) = \Vs_j^{-\im \alpha}\oplus \Vmd_i \oplus\Vs_j^{\im \alpha} \andq \head(\Ps_i)=\Vmd_i.
\]
It is shown in \cite{X3} that $\Pmd_i \cong \Vmd_0 \otimes \Ps_i$, where $\Ps_i := \Ps_i^1$.

We call a weight vector $v$ \textit{dominant} if $v \in \ker(FE)^2$, following the terminology of \cite{CGP3}. Recall that $j = r - 2 -i$. The modules $\Ps_i$ have a basis consisting of $2r$ weight vectors
\[
\left\{ w_m^H, w_m^S \right\}_{0 \leq m \leq i} \cup \left\{ w_m^L, w_m^R \right\}_{0 \leq m \leq j},
\]
where $w_0^H$ is the dominant vector generating the module. The basis vectors satisfy the relations
\begin{equation}\label{eq:proj_rels1}
\begin{alignedat}{4}
w_m^H &= F^m \, w_0^H 					&&& 		& &\forq m=1,\ldots,i,
\\
w_j^R 	&= E \, w_0^H		&\andq &&w_{j-m}^R &= E^m \, w_j^R &\forq m=1,\ldots,j,
\\
w_0^S 		&= F \, w_j^R 		&\andq &&w_m^S &= F^m \, w_0^S 		&\forq m=1,\ldots,i,
\\
w_0^L 	&= F^{i+1} \, w_0^{H} &\andq &&w_m^L &= F^m \, w_0^L &\forq m=1,\ldots,j.
\end{alignedat}
\end{equation}
and
\begin{equation}\label{eq:proj_rels2}
\begin{alignedat}{3}
k \, w_m^L 	&= q^{-(i+2+2m)/2} \, w_m^L,  & k \, w_m^R 	&= q^{(j+r-2m)/2} \, w_m^R,
\\
E \, w_m^R 			&= w_{m-1}^R									& k \, w_m^X 	&= q^{(i-2m)/2} \, w_m^X,  \; \text{ for } X \in \{H, S\}
\\
E \, w_m^H 	&= \gamma_{i,m} \, w_{m-1}^H + w_{m-1}^S, 	& E \, w_0^L &= w_i^S,
\\
E \, w_m^S 	&= \gamma_{i,m} \, w_{m-1}^S,					& E \, w_m^L 	&= -\gamma_{j,m} \, w_{m-1}^L, 	
\\
F \, w_i^H 	&= w_0^L,				&E \, w_0^R 		= E \, w_0^S &= F \, w_i^S = F \, w_j^L = 0,		
\\
F \, w_m^R 	&= -\gamma_{j,m} \, w_{m-1}^R,		&F \, w_m^X &= w_{m-2} ^X, \;  X \in \{L, H, S\},
\end{alignedat}
\end{equation}
where $\gamma_{m,m'} = \qint{m'} \qint{m-m'+1} = \gamma_{m,m-m'+1}$ and the indices of the vectors are such that the vectors are defined. Note that the isomorphism $\soc(\Ps_i) \cong \Vs_i$ is given by $w_m^S \mapsto v_m^1$, where $\{v _0^1, \ldots, v_i^1 \}$ is the basis $\Vs_i$ of Remark~\ref{rem:V_basis_F}.

Let $x_{\alpha,i} \colon \Pmd_i \to \Pmd_i$ be the nilpotent map of order $2$ defined by $u_0^{\alpha} \otimes w_0^H \mapsto u_0^{\alpha} \otimes w_0^S$, which generates $\End_U(\Pmd_i)$. Hence, $\End_U(\Pmd_i)$ is isomorphic to the algebra of dual numbers $\C[x_{\alpha,i}]/(x_{\alpha,i}^2)$.
Moreover
\[
F E \, (u_0^{\alpha} \otimes w_0^H) = \alpha^2 (u_0^{\alpha} \otimes w_0^S)
\]
which, together with $\alpha^4=1$, implies
\begin{equation}\label{eq:xmi}
x_{\alpha,i} (u_0^{\alpha} \otimes w_0^H) = \alpha^2 \, FE (u_0^{\alpha} \otimes w_0^H).
\end{equation}
Further, note that there exists a bijection
\begin{equation} \label{eq:proj_bij}
\begin{aligned}
\widetilde{x}_{\alpha,i} \colon \head(\Pmd_i) &\to \soc(\Pmd_i)
\\
\left( u_0^\alpha \otimes w_0^H \right)_m &\mapsto \left( u_0^\alpha \otimes w_0^S \right)_m \forq 0 \leq m \leq i.
\end{aligned}
\end{equation}

\bigbreak
Since the module $\Vs_{r-1}$ is projective, any tensor power of $\Vs_{r-1}$ decomposes to a direct sum of indecomposable projective modules. Therefore, due to the above classification of indecomposable projective modules in $D$-mod \cite{X3}, we can write any tensor power of $\Vs_{r-1}$ as a direct sum of the modules $\Vmd_{r-1}$ and the modules $\Pmd_i$.

\subsection{Strong weights}
Let $V$ be a weight module of $D$. Let $-2r \leq \lambda, \lambda' < 2r$ and let $v$ and $w$ be weight vectors of $V$, such that $k v = q^{\lambda/2} v$ and $k w = q^{\lambda'/2} w$. We define the operator $H \colon V \to V$ by:
\begin{equation} \label{eq:H}
H v = \lambda \, v.
\end{equation}
We call $\lambda$ the \textit{strong weight} of $v$. Further, we define for every $n \in \Npp$ an operator $H_n \colon V^{\otimes n} \to V^{\otimes n}$ by
\begin{equation} \label{eq:Hn}
H_n (v_1 \otimes \ldots \otimes v_n) := \sum_{i=1}^n v_1 \otimes \ldots \otimes v_{i-1} \otimes H v_i \otimes v_{i+1} \otimes \ldots v_n,
\end{equation}
where $v_1, \ldots, v_n \in V$. Let $\lambda_i$ be the strong weight of $v_i$, for $1 \leq i \leq n$. Then by definition, it holds
\[
H_n (v_1 \otimes \ldots \otimes v_n) = \left( \lambda_1 + \ldots + \lambda_n \right) v_1 \otimes \ldots \otimes v_n.
\]
We call similarly the scalar $\lambda_1 + \ldots + \lambda_n \in \C$ the \textit{strong weight} of $v_1 \otimes \ldots \otimes v_n$.

Finally, let $q^{H \otimes H/2} \colon V \otimes V \to V \otimes V$ be the operator defined by
\begin{equation} \label{eq:qHH2}
q^{H \otimes H/2} (v \otimes w) = q^{\lambda \lambda'/2} v \otimes w.
\end{equation}

\subsection{Ribbon structure of \texorpdfstring{$D$}{D}}
In contrast to $U$-mod, the category $D$-mod is ribbon, since $D$ is a ribbon Hopf algebra. The ribbon structure is given by
\begin{itemize}
\item the \textit{braiding operator} $c_{V,W} \colon V \otimes W \to W \otimes V$ defined by
\[
v \otimes w \mapsto \tau \circ R (v \otimes w),
\]
where $\tau(x \otimes y) = y \otimes x$ is the flip map;
\item the \textit{twist operator} $\theta_V \colon V \to V$ defined by $v \mapsto \theta^{-1} v$.
\end{itemize}
The duality maps for the pivotal structure of $D$ can be found in \cite{BBG}.

The action of the R-matrix \eqref{eq:Rmatrix} can be given by the formula of the R-matrix of quantum $\sltwo$ (see \cite{JK} and \cite{CGP3}) as follows.

\begin{lemma} The action of the R-matrix of $D$ on any weight $D$-module $V$ is given by
\begin{equation} \label{eq:Rmatrix_unrolled}
R = q^{H \otimes H/2} \sum_{n=0}^{r-1} \frac{\qnum{1}^{2n}}{\qnum{n}!} q^{n(n-1)/2}  \left( E^n  \otimes F^n \right).
\end{equation}
\end{lemma}
The proof can be found in Appendix~\ref{app:Rmatrix_weight}. Therefore, since the ribbon element is defined by $\theta = m ( S \otimes \Id) R_{21}$ also the action of the twist operator of $D$ \eqref{eq:twist} on a $D$-module is given by the formula for the twist of quantum $\sltwo$. Moreover, the relations \eqref{eq:proj_rels1} and \eqref{eq:proj_rels2} for the projective module $\Ps_i$ coincide with the ones for the projective modules found in \cite{CGP3} (by considering the action of $K$ instead of $k$ on the weight vectors). Therefore, since $\theta$ is central, its action on the indecomposable projective modules $\Ps_i$ can be expressed in terms of $\End(\Ps_i)$ \cite[Lemma 6.10]{CGP3}:
\begin{equation} \label{eq:twist_proj}
\theta_{\Ps_i} = (-1)^i q^{\frac{i^2 + 2 i}{2}} \left( I_{1, i} - (r - i - 1) \frac{\qnum{1}^2}{\qnum{i+1}} x_{1 ,i} \right),
\end{equation}
where $I_{\alpha,i}$ is the identity endomorphism of $\Pmd_i$.

\section{Specializations of the Lawrence representations at roots of unity} \label{sec:WH}

The aim of this section is to adapt the technique of \cite{JK} in order to study the $B_n$-action on the tensor power $\Vs_{r-1}^{\otimes n}$ and to recover, as expected, specializations of the Lawrence representations, with the two variables fixed at roots of unity depending on $q$. We denote these representations by $\WH_{n,\ell}$. Moreover, we identify the basis vectors of $\WH_{n,\ell}$ with vectors in the projective modules of the direct sum decomposition of $\Vs_{r-1}^{\otimes n}$. Then using this identification, we prove that the representations $\WH_{n,\ell}$ are not simple if $n$, $\ell$ and $r$ satisfy a certain modular condition.
\subsection{\texorpdfstring{$B_n$}{Bn}-action on strong weight spaces}
We define a $B_n$-action on $\Vs_{r-1}^{\otimes n}$, for $n \in \Npp$, via the $R$-matrix \eqref{eq:Rmatrix}. Let the map
\begin{equation} \label{eq:Bn_action_V}
\begin{aligned} 
\Rs: \Vs_{r-1} \otimes \Vs_{r-1} &\rightarrow \Vs_{r-1} \otimes \Vs_{r-1},
\\
v \otimes w &\mapsto q^{-\frac{(r - 1)^2}{2}} \, c_{\Vs_{r-1},\Vs_{r-1}}(v \otimes w) = q^{-\frac{(r - 1)^2}{2}} \, \tau \circ R (v \otimes w).
\end{aligned}
\end{equation}
The normalization factor $q^{-\frac{(r - 1)^2}{2}}$ cancels out a constant appearing by the action of $q^{H \otimes H / 2}$ on $\Vs_{r-1} \otimes \Vs_{r-1}$. In detail, for $0 \leq m, m' \leq r-1$ we have:
\[
q^{H \otimes H / 2} (u_m \otimes u_{m'}) = q^{\frac{(r-1-2m)(r-1-2m')}{2}} (u_m \otimes u_{m'})= q^{\frac{(r-1)^2}{2}} q^{\frac{-(m+m')(r-1) + m m'}{2}} (u_m \otimes u_{m'}).
\]
Substituting Equation~\eqref{eq:Rmatrix_unrolled} in the definition of $\Rs$, we can calculate the action of $\Rs$ with respect to the basis $\{u_0, \ldots, u_{r-1}\}$ of $\Vs_{r-1}$ (see Appendix~\ref{app:Rmatrix_V}):
\begin{equation} \label{eq:Rmatrix_V}
\begin{aligned}
&\Rs (u_i \otimes u_j) =
\\
&s^{-(i+j)} \sum_{n=0}^{\min (i,r-j-1)} q^{2(i-n)(j+n)} q^{n(n-1)/2} \qbin{n+j}{j} \prod_{m=0}^{n-1} \qnum{m+j+1} \, u_{j+n} \otimes u_{i-n}.
\end{aligned}
\end{equation}
Moreover, the maps
\begin{equation} \label{eq:sigma_i}
\sigma_i = \1^{\otimes i-1} \otimes \Rs \otimes \1^{\otimes n-i-1},
\end{equation}
define a $B_n$-representation on $\Vs_{r-1}^{\otimes n}$. They satisfy the braid relations by construction and they also commute with the action of $D$.

We now define strong weight subspaces of $\Vs_{r-1}^{\otimes n}$.
Let $\varepsilon_1, \ldots, \varepsilon_n  \in \N$, such that $0 \leq \varepsilon_i < r$, for all $1 \leq i \leq n$. The action of the operator $H_n \colon \Vs_{r-1}^{\otimes n} \to \Vs_{r-1}^{\otimes n}$ \eqref{eq:Hn} is given by
\begin{align*}
H_n (u_{\varepsilon_1} \otimes \ldots \otimes u_{\varepsilon_n})
=\,& \left( n(r-1)-2(\varepsilon_1 + \ldots + \varepsilon_n) \right) u_{\varepsilon_1} \otimes \ldots \otimes u_{\varepsilon_n}.
\end{align*}
We now define the following.

\begin{defn} \label{def:Vnl}
Let $n, r \in \Npp$ and $\ell \in \N$.
The \textit{strong weight space of strong weight \mbox{$n(r-1) - 2 \ell$}} of $\Vs_{r-1}^{\otimes n}$ is defined by
\begin{align*}
\VH_{n,\ell} :=& \ker \left( H_n - \left( n(r-1) - 2 \ell \right) \right)
\\
=& \spn \left\{ u_{\varepsilon_1} \otimes u_{\varepsilon_2} \otimes \ldots \otimes u_{\varepsilon_n} \in \Vs_{r-1}^{\otimes n} \, \mid \, \varepsilon_1+ \cdots + \varepsilon_n = \ell \right\}
 \subset \Vs_{r-1}^{\otimes n}.
\end{align*}
\end{defn}

Note that the vectors of $\VH_{n,\ell}$ have weight $q^{\frac{n(r-1)}{2} - \ell} = s^{n/2} q^{-\ell}$. Therefore, if $\ell, \ell' \in \N$ such that $\ell \equiv \ell' \mod 2r$, the spaces $\VH_{n,\ell}$ and $\VH_{n,\ell'}$ are subspaces of the $k$-eigenspace of weight $s^{n/2} q^{-\ell}$. Therefore, Definition~\ref{def:Vnl} and the use of strong weights provide indeed a finer decomposition than the one into $k$-eigenspaces.

For every $\ell$ the space $\VH_{n,\ell}$ is closed under the $B_n$-action \eqref{eq:sigma_i} and therefore, is also a $B_n$-representation. For $\ell=0$ we obtain an one-dimensional strong weight space spanned by the highest strong weight vector $u_0^{\otimes n} \in \Vs_{r-1}^{\otimes n}$. On the other hand, for $\ell = n(r-1)$ the strong weight space is spanned by the lowest strong weight vector of $u_{r-1}^{\otimes n} \in \Vs_{r-1}^{\otimes n}$. Both of these spaces are the trivial $B_n$-representation. For $0 < \ell < n(r-1)$ we get spaces of higher dimension, from which we extract more interesting $B_n$-representations. Finally, it is evident that $\VH_{n,\ell} = \{ 0 \}$ for $\ell > n(r-1)$.

We proceed by defining the corresponding highest strong weight spaces.
\begin{defn} 
Let $n, r \in \Npp$ and $\ell \in \N$. The \textit{highest strong weight space} corresponding to the strong weight $n(r-1) - 2 \ell$ is defined by
\[
\WH_{n,\ell} := \ker(E) \cap \VH_{n,\ell}.
\]
\end{defn}
The space $\WH_{n,\ell}$ is also a $B_n$-representation, since the action of $\Rs$ intertwines the $D$-action.

\subsection{The Jackson-Kerler construction} In this section, we give a basis for the space $\VH_{n,\ell}$ following \cite{JK}. Define now the following two sets for $n, \ell \geq 2$:
\begin{equation}\label{eq:V_nl_basis}
\begin{aligned}
\mathcal{A}_{n,\ell} &= \{ u_{\varepsilon_1} \otimes u_{\varepsilon_2} \otimes \ldots \otimes u_{\varepsilon_n} \in \VH_{n,\ell} \, \mid \, \exists k \text{ s.t. }\varepsilon_k = 1 \text{ and } \varepsilon_j=0 \; \forall \, j<k \},
\\
\mathcal{B}_{n,\ell} &= \{ u_{\varepsilon_1} \otimes u_{\varepsilon_2} \otimes \ldots \otimes u_{\varepsilon_n} \in \VH_{n,\ell} \, \mid \, \exists k \text{ s.t. }\varepsilon_k > 1 \text{ and } \varepsilon_j=0 \; \forall \, j<k \},
\end{aligned}
\end{equation}
and set for all $\ell \geq 0$:
\[
\AH_{n,\ell} = \spn(\mathcal{A}_{n,\ell}) \quad \text{and} \quad \BH_{n,\ell} = \spn(\mathcal{B}_{n,\ell}).
\]
For $\ell=1$ and $\ell=0$ the definitions of $\mathcal{A}_{n,\ell}$ and $\mathcal{B}_{n,\ell}$ are slightly modified:
\[
\begin{array}{rlcrl}
\mathcal{A}_{n,1} &= \{ c_i \mid 1 \leq i \leq n -1 \} &\andq& \mathcal{B}_{n,1} &= \{ c_n\},
\\
\mathcal{A}_{n,0} &= \{ u_0^{\otimes n } \} &\andq& \mathcal{B}_{n,0} &= \emptyset.
\end{array}
\]
where
\begin{equation} \label{eq:ci}
c_i :=  u_0^{\otimes i -1} \otimes u_1 \otimes u_0^{\otimes n-i} \in \VH_{n,1}.
\end{equation}
It is evident that $\VH_{n,\ell} = \AH_{n,\ell} \oplus \BH_{n,\ell}$. 

For any $x \in \N$ a \textit{$p$-composition} of $x$ is a tuple $\vec{y}= (y_1, \ldots, y_p)$ such that $y_i \in \Np$, for all $1 \leq i \leq p$, and $y_1 + \ldots + y_p = x$. A \textit{weak $p$-composition} of $x$ is a tuple $\vec{y}$ such that $y_i \in \N$, for all $1 \leq i \leq p$, and $y_1 + \ldots + y_p = x$. Moreover, for $j>1$ let $\vec{\varepsilon} = (\varepsilon_j, \ldots, \varepsilon_n)$ be a weak composition of $\ell-1$. Then any element of $\mathcal{A}_{n,\ell}$ can be written as
\begin{equation} \label{eq:a_epsilon}
a_{\vec{\varepsilon}} := u_0^{\otimes j -2} \otimes u_1 \otimes u_{\vec{\varepsilon}},
\quad \text{ where } u_{\vec{\varepsilon}} = u_{\varepsilon_j} \otimes \ldots \otimes u_{\varepsilon_n} \in \VH_{n-j+1,\ell-1}.
\end{equation}

\begin{rem} \label{rem:Rmatrix_formula}
Let $u_{\varepsilon_1} \otimes \ldots \otimes u_{\varepsilon_n} \in \VH_{n,\ell}$. Then by Definition~\ref{def:Vnl} we have $\varepsilon_k + \varepsilon_{k'} \leq \ell < r$ for all $1 \leq k,k' \leq n$. So, for all $0 \leq i \leq n-1$ we have $\varepsilon_i + \varepsilon_{i+1} < r \Leftrightarrow \varepsilon_i < r - \varepsilon_{i+1}$ and hence, $\min (\varepsilon_i,r-\varepsilon_{i+1}-1) = \varepsilon_i$. Therefore, \eqref{eq:Rmatrix_V} can be written as:
\begin{align*}
&\Rs (u_i \otimes u_j) =
s^{-(i+j)} \sum_{n=0}^{i} q^{2(i-n)(j+n)} q^{n(n-1)/2} \qbin{n+j}{j} \prod_{m=0}^{n-1} \qnum{m+j+1} \, u_{j+n} \otimes u_{i-n}.
\end{align*}
The above formula coincides with the formula for the action of the $R$-matrix on the Verma module of quantum $\sltwo$ \cite[Eq. 22]{JK} (by substituting the variables $q$ and $s$ with $q= e^{\pi \im /r}$ and $s=q^{r-1}$). Therefore the calculations involving the action of $\Rs$ on $\VH_{n,\ell}$ (and hence for the $B_n$-action on $\VH_{n,\ell}$) remain the same as in \cite{JK}.
\end{rem}

\subsection{Dimension of the spaces \texorpdfstring{$\VH_{n,\ell}$}{VH}} \label{sec:dims}
In \cite{JK} the corresponding space to $\VH_{n,\ell}$, denoted here by $\Vb_{n,\ell}$, is defined similarly, except for the fact that the indices $\varepsilon_1, \ldots, \varepsilon_n$ are any elements of $\N$. Thus, $\dim \Vb_{n,\ell}$ equals $N(\ell,n) = \binom{n + \ell - 1}{n-1}$, that is, the number of weak $n$-compositions of $\ell$, or equivalently the number of $n$-compositions of $\ell+n$ (see also \cite[Section 1.2]{Sta}). Let now $\kappa(\ell,r,n)$ be the number of \textit{weak $n$-compositions of $\ell$ with each part less than $r$}. By \cite[Section 1, Exercise 28]{Sta} we have
\[
\kappa(\ell,r,n) = \sum_{\substack{t+sr=\ell\\(t,s)\in \N^2}} (-1)^s \binom{n+t-1}{t} \binom{n}{s}.
\]
Note that, for $0 \leq \ell < r$, it holds $\kappa(\ell,r,n) = N(\ell,n)$ and hence, $\dim\VH_{n,\ell} = \dim \Vb_{n,\ell}$, for $0 \leq \ell < r$. In our case, due to the finiteness of $\Vs_{r-1}$ we have the following.

\begin{lemma} \label{lem:dim_Vnl}
Let $n, r \in \Npp$ and $\ell \in \N$. It holds that
\[
\dim \VH_{n,\ell} = \kappa(\ell,r,n).
\]
\end{lemma}

For example, for $\ell = r$, we can easily observe that $\dim \VH_{n,r} < \dim \Vb_{n,r}$. By Lemma~\ref{lem:dim_Vnl} we get:
\[
\dim \VH_{n,r} = \kappa(r,r,n) = \binom{n+r-1}{r} - n = \dim \Vb_{n,r} - n,
\]
since the only integer solutions to the equation $t + sr = r$ are $(t,s) = (r,0)$ and $(t,s) = (0,1)$. The $n$ vectors allowed in $\Vb_{n,r}$ but not in $\VH_{n,r}$ are of the form $u_0^{i-1} \otimes u_r \otimes u_0^{n-i} \in \Vb_{n,r}$, with $1 \leq i \leq n$. Furthermore, it holds that $\dim \VH_{n,n(r-1)-\ell} = \dim \VH_{n,\ell}$, since there is a bijection between weak $n$-compositions of $\ell$ with parts less than $r$ and weak $n$-compositions of $n(r-1)-\ell$ with parts less than $r$.

We now prove the following for $\ell < r$.

\begin{lemma}
Let $n, r \in \Npp$ and $\ell \in \N$ such that $\ell < r$. It holds that
\begin{align}
\dim \AH_{n,\ell} &= \binom{n + \ell - 2}{n-2}, \label{eq:dim_Anl}
\\
\dim \BH_{n,\ell} &= \VH_{n,\ell-1}. \label{eq:dim_Bnl}
\end{align}
\end{lemma}

\begin{proof}
Let $u_0^{\otimes i-1} \otimes u_1 \otimes u_{\varepsilon_{i+1}} \otimes \ldots \otimes u_{\varepsilon_{n}} \in \AH_{n,\ell}$ for $\ell < r$ and some $1 \leq i \leq n$. Then $\varepsilon_{i+1} + \ldots + \varepsilon_{n-1} = \ell - 1$. Therefore, $u_0^{\otimes i-1} \otimes u_1 \otimes u_{\varepsilon_{i+1}} \otimes \ldots \otimes u_{\varepsilon_{n}}$ corresponds to a weak $(n-i)$-composition of $\ell - 1$ with parts less than $r$. So:
\begin{align*}
\dim \AH_{n,\ell} &= \sum_{i=1}^n \kappa(\ell-1, r, n-i) 
\stackrel{\ell < r}{=} \sum_{i=1}^n \binom{n + \ell - i - 2}{n - i - 1}
\\
&= \sum_{i=1}^n \binom{n + \ell - i - 2}{\ell- 1}
= \sum_{m=\ell-2}^{n- \ell -3} \binom{m}{\ell - 1} = \binom{n + \ell - 2}{n-2}.
\end{align*}
Finally
\[
\dim \BH_{n,\ell} = \dim \VH_{n,\ell} - \dim \AH_{n,\ell} = \binom{n + \ell - 2}{n-1} = \dim \VH_{n,\ell-1}. \qedhere
\]
\end{proof}

\subsection{The highest strong weight spaces \texorpdfstring{$\WH_{n,\ell}$}{WH}}
In this subsection we give a basis for the space $\WH_{n,\ell}$, when $\ell < r$, following \cite{JK}. By \eqref{eq:dim_Bnl} the dimensions of the spaces $\BH_{n,\ell}$ and $\VH_{n,\ell-1}$ are equal for $\ell < r$. Moreover, as the following lemma shows, these two spaces are isomorphic as $\C$-vector spaces.

\begin{lemma} \label{lem:E_iso}
Let $n, r \in \Npp$ and $\ell \in \N$. The map $E \vert_{\BH_{n,\ell}}: \BH_{n,\ell} \rightarrow \VH_{n,\ell-1}$ is a $\C$-vector space isomorphism for $1 \leq \ell < r$. In detail, it is injective for all $\ell \geq 1$ and surjective for all $1 \leq \ell < r$.
\end{lemma}

\begin{proof}
For the proof we modify the proof of \cite[Lemma 8]{JK}.

We need to prove that for every $u_{\vec{\varepsilon}} = u_{\varepsilon_1} \otimes \ldots \otimes u_{\varepsilon_n} \in \VH_{n,\ell-1}$, there exists some $b \in \BH_{n,\ell}$ such that $E \, b = u_{\vec{\varepsilon}}$. Let $\varepsilon_m$ be the first non-zero index in $\vec{\varepsilon} = (\varepsilon_1, \ldots, \varepsilon_n)$ and set $j := \ell - \varepsilon_m$. We apply induction on $j$. When $j=1$, it holds that $\varepsilon_m = \ell-1$ and we have that $E (u_0^{\otimes m-1} \otimes u_\ell \otimes u_0^{\otimes n-m}) = u_0^{\otimes m-1} \otimes u_{\ell-1} \otimes u_0^{\otimes n-m}$.

For the induction step we take $u_{\vec{\varepsilon}} = u_0^{\otimes m-1} \otimes u_{\varepsilon_{m}} \otimes \ldots \otimes u_{\varepsilon_{n}} \in \VH_{n,\ell-1}$ such that $\ell - \varepsilon_m = j + 1$. Set $b = u_0^{\otimes m-1} \otimes u_{\varepsilon_{m}+1} \otimes \ldots \otimes u_{\varepsilon_{n}} \in \BH_{n,\ell-1}$. Such an element exists in $\BH_{n,\ell-1}$ if and only if $\varepsilon_{m}+1 < r$, or equivalently, $\ell - j < r$. This condition holds for any $j \in \Np$ if and only if $\ell < r$. By assuming that the condition $\ell < r$ holds, the element $b \in \BH_{n,\ell-1}$ exists, so we have that:
\[
E \, b = \zeta \, u_{\vec{\varepsilon}} + \text{(other terms)}, \quad \text{ with } \zeta \in \C.
\]
Now $E$ is surjective on the ``other terms'' of the relation, since they satisfy the induction hypothesis (the first non-zero index on all of them is $\varepsilon_m + 1$). Hence, $E \vert_{\BH_{n,\ell}}: \BH_{n,\ell} \rightarrow \VH_{n,\ell-1}$ is surjective for $1 \leq \ell < r$.

It remains to show that $\ker E \vert_{\BH_{n,\ell}} = \{0\}$. Let $b$ be a non-zero element in $\BH_{n,\ell}$. In the expression of $b$ there exists a minimal term $u_{\varepsilon_1} \otimes \ldots \otimes u_{\varepsilon_m} \otimes \ldots \otimes u_{\varepsilon_n}$ such that $\varepsilon_i = 0$ for all $i < m$ and $2 \leq \varepsilon_m < r$ and if $u_{\varepsilon'_1} \otimes \ldots \otimes u_{\varepsilon'_m} \otimes \ldots \otimes u_{\varepsilon'_n}$ is contained in the expression of $b$, then $\varepsilon'_i = 0$ for all $i < m$ and either $\varepsilon'_m=0$ or $\varepsilon_m \leq \varepsilon'_m < r$. Applying the operator $E$ on $b$, we obtain an expression of $E \, b$ containing the term $u_{\varepsilon_1} \otimes \ldots \otimes u_{\varepsilon_{m-1}} \otimes \ldots \otimes u_{\varepsilon_n}$. This term cannot be canceled by some other term in $E \, b$ (due to minimality), hence $E b \neq 0$ and finally, $\ker E \vert_{\BH_{n,\ell}} = \{0\}$ for every $1 \leq \ell$.
\end{proof}

Let $a_{\vec{\varepsilon}} \in \AH_{n,\ell}$ as in \eqref{eq:a_epsilon}, that is:
\[
a_{\vec{\varepsilon}} = u_0^{\otimes j -2} \otimes u_1 \otimes u_{\vec{\varepsilon}},
\quad \text{ where } u_{\vec{\varepsilon}} = u_{\varepsilon_j} \otimes \ldots \otimes u_{\varepsilon_n} \in \VH_{n-j+1,\ell-1},
\]
for a $2 \leq j \leq n$.
Due to the above isomorphism, $\ker E$ can be parametrized for $1 \leq \ell < r$ by the space $\AH_{n,\ell}$, using the following map \cite{JK}:
\begin{equation} \label{eq:map_phi}
\begin{aligned}
\Phi: \VH_{n,\ell} &\rightarrow \VH_{n,\ell} \\
a_{\vec{\varepsilon}} &\mapsto  \sum_{m=0}^{\ell} b_{\vec{\varepsilon},m} \, u_0^{\otimes j - 2} \otimes u_m \otimes E^{m-1} \, u_{\vec{\varepsilon}} &\text{ for } a_{\vec{\varepsilon}} \in \mathcal{A}_{n,\ell}\\
b &\mapsto b &\text{ for } b \in \mathcal{B}_{n,\ell},
\end{aligned}
\end{equation}
where
\[
b_{\vec{\varepsilon},m} = (-1)^{m-1} s^{(m-1)(j-n-1)} q^{(m-1)(2\ell-m-2)},
\]
and where $E^{-1} \, a_{\vec{\varepsilon}}$ denotes the unique element in $\BH_{n-j+1,\ell}$ such that $E(E^{-1} \, a_{\vec{\varepsilon}}) = a_{\vec{\varepsilon}}$. This element exists due to the isomorphism between $\BH_{n,\ell}$ and $\VH_{n,\ell-1}$ for $1 \leq \ell < r$. Furthermore, we have the following statement.

\begin{lemma} \label{lem:Phi_auto}
Let $n, r \in \Npp$ and $\ell \in \N$ such that $0 \leq \ell < r$. The map $\Phi$ is an automorphism of $\VH_{n,\ell}$ with $(\Phi - \1)^2 = 0$.
\end{lemma}

The proof of Lemma~\ref{lem:Phi_auto} is similar to the one of \cite[Lemma 9]{JK} and can be found in Appendix~\ref{app:Phi_auto}. Therefore, $\Phi$ is a change of basis map on $\VH_{n,\ell}$, under which the space $\WH_{n,\ell} = \ker E \cap \VH_{n,\ell}$ can be parametrized by $\AH_{n, \ell}$ for $0 \leq \ell < r$. In detail, we have the following:

\begin{lemma} \label{lem:Wnl_Anl}
Let $n, r \in \Npp$ and $\ell \in \N$ such that $1 \leq \ell < r$.
The map $E \circ \Phi$ vanishes on $\AH_{n,\ell}$ and is injective on $\BH_{n,\ell}$ with:
\[
E \circ \Phi = 0 \oplus E \vert_{\BH_{n,\ell}} : \AH_{n,\ell} \oplus \BH_{n,\ell} \rightarrow \VH_{n,\ell-1}
\]
and hence $\Phi$ is an isomorphism of vector spaces: $\Phi: \AH_{n,\ell} \xrightarrow{\cong} \WH_{n,\ell}$.
\end{lemma}

The proof is analogous to the one of \cite[Lemma 10]{JK} and can be found in Appendix~\ref{app:Wnl_Anl}. Finally, we state the main Theorem of this section.

\begin{thm} \label{thm:Wnl_exist}
Let $n, r \in \Npp$ and $\ell \in \N$ such that $0 \leq \ell < r$. The highest strong weight space $\WH_{n,\ell}$ is a $\C$-vector space of dimension $d_{n,\ell} := \binom{n+\ell-2}{\ell}$ 
and defines a $B_n$-representation given by a homomorphism:
\[
\rho_{n,\ell}^{\WH} \colon B_n \to \mathrm{GL} \left( d_{n,\ell}, \C \right).
\]
Further, the representation is isomorphic to a specialization of the Lawrence representation $\Wb_{n,\ell}$.
\end{thm}

\begin{proof}
For $\ell=0$ we obtain clearly the trivial $B_n$-representation since $\dim \WH_{n,0} = \dim \VH_{n,0} = 1$. Suppose that $1 \leq \ell < r$.
By Lemma~\ref{lem:Wnl_Anl} $\WH_{n,\ell}$ is a $\C$-vector space of dimension $d_{n,\ell}$. Due to Remark~\ref{rem:Rmatrix_formula} the calculations remain the same as in \cite{JK}, with the difference that the variables $s$ and $q$ are specialized to the complex values $q=e^{\pi \im /r}$ and $s=q^{r-1}$. Hence, we obtain a specialization of the corresponding Lawrence representation, since the Lawrence representations are isomorphic to the representations $\Wb_{n,\ell}$ \cite{Ito2}.
\end{proof}

\subsection{Structure of the representations \texorpdfstring{$\WH_{n,\ell}$}{WH}}
The goal of this subsection is to represent the vectors of $\WH_{n,\ell}$, for $\ell < r$, by vectors in the projective modules in the direct sum decomposition of $\Vs_{r-1}^{\otimes n}$.

The vectors of $\WH_{n,\ell}$ are in $\ker E$. Therefore, by the structure of the projective modules of $D$, the vectors of $\WH_{n,\ell}$ are represented by
\begin{enumerate}[label=(\roman{*})]
\item \label{item:C} highest weight vectors $u_0^\alpha$ of a simple module $\Vmd_{r-1}$, for some $\alpha \in \{ \pm 1, \pm \im \}$;
\item \label{item:S} socle vectors $u_0^\alpha \otimes w_0^S$ in a projective module $\Pmd_i$, for some $0 \leq i < r-1$ and some $\alpha \in \{ \pm 1, \pm \im \}$;
\item \label{item:R} and finally, right-most highest weight vectors $u_0^{- \im \alpha} \otimes w_0^R \in \Ps_j^{- \im \alpha}$, where $j=r-2-i$.
\end{enumerate}

We define the vector subspaces $\CH_{n,\ell}$, $\SH_{n,\ell}$ and $\RH_{n,\ell}$ as the spaces spanned by vectors represented as in \ref{item:C}, \ref{item:S} and \ref{item:R} respectively. Therefore, there exists a decomposition into vector spaces
\[
\WH_{n,\ell} = \CH_{n,\ell} \oplus \SH_{n,\ell} \oplus \RH_{n,\ell}.
\]

First, we split the $B_n$-module $\WH_{n,\ell}$ as the direct sum of $B_n$-modules $\CH_{n,\ell}$ and $\SH_{n,\ell} \oplus \RH_{n,\ell}$. By \eqref{eq:action_Va} we have $E^{r-1} F^{r-1} u_0^\alpha = \alpha^{2r} (\qint{r-1}!)^2 u_0^\alpha$. Further, by \eqref{eq:proj_rels1} and \eqref{eq:proj_rels2} we have $u_0^\alpha \otimes w_0^S, u_0^{- \im \alpha} \otimes w_0^R \in \ker (E^{r-1} F^{r-1})$. Since the $B_n$-action intertwines the quantum group action, there exists a split short exact sequence of $B_n$-modules
\[
\begin{tikzcd}
0 \arrow[r] & \SH_{n,\ell} \oplus \RH_{n,\ell} \arrow[r, hook] & \WH_{n,\ell} \arrow[r, "E^{r-1} F^{r-1}"] & \CH_{n,\ell} \arrow[r] & 0,
\end{tikzcd}
\]
where the section for $E^{r-1} F^{r-1}$ is the inclusion $\CH_{n,\ell} \hookrightarrow \WH_{n,\ell}$ composed by the multiplication by $\alpha^{2r} (\qint{r-1}!)^{-2}$.

Note that, the direct sum of vector spaces $\SH_{n,\ell} \oplus \RH_{n,\ell}$ is not necessarily a direct sum of $B_n$-modules. In $D$-mod there are non trivial maps between the projective modules $\Pmd_i$ and $\Ps_j^{- \im \alpha}$ \cite{X2} (see also \cite{CGP3}).
Later in this section we give an explicit example (Lemma~\ref{lem:RS_non_split}).

\bigbreak

Let $w \in \WH_{n,\ell}$. Recall that $w$ has strong weight $n(r-1) - 2\ell$.
By writing $\alpha = q^{mr/2}$, for some $m \in \Z$, we have by \eqref{eq:proj_rels2} that
\[
\begin{array}{rcll}
k w &=& q^{(mr+r-1)/2} \, w, &\text{ for } w \in \CH_{n,\ell},
\\
k w &=& q^{(mr+i)/2} \, w, &\text{ for } w \in \SH_{n,\ell} \oplus \RH_{n,\ell} \text{ and where } 0 \leq i < r-1.
\end{array}
\]
Therefore, $w$ has strong weight $4 \kappa r + mr+i$, for some $\kappa \in \Z$ such that $-2r \leq 4 \kappa r + mr+i < 2r$, with $i=r-1$ if $w \in \CH_{n,\ell}$ and $0 \leq i < r-1$ otherwise.
Hence
\begin{equation}\label{eq:struct0}
\begin{aligned}
\CH_{n,\ell} \neq \{ 0 \} &\Leftrightarrow r-1 \equiv n + 2(\ell-1)\mod r,
\\
\SH_{n,\ell} \oplus \RH_{n,\ell} \neq \{ 0 \} &\Leftrightarrow j \equiv n + 2(\ell-1) \mod r, \text{ for some } 0 \leq j < r-1.
\end{aligned}
\end{equation}
Recall that the integer $j$ corresponds to the projective module $\Ps_j^{-\im \alpha}$ containing the vectors of $\RH_{n, \ell}$.
It remains to investigate under which conditions $w \in \SH_{n, \ell}$ and/or $w \in \RH_{n,\ell}$. We prove in this section the following theorem.

\begin{thm} \label{thm:Wnl}
Let $n, r \in \Npp$ and $\ell \in \N$ such that $\ell < r$. Let $0 \leq j \leq r-1$ such that $j \equiv n + 2(\ell-1) \mod r$. Then we have the following.
\begin{enumerate}
\item If $j=r-1$, then $\WH_{n,\ell} = \CH_{n,\ell}$.
\item If $j \geq \ell$, then $\WH_{n,\ell} = \RH_{n,\ell}$.
\item If $j < \ell$ and $n \geq 3$, then $\WH_{n,\ell} = \SH_{n,\ell} \oplus \RH_{n,\ell}$. Moreover, $\WH_{n,\ell}$ is not simple with $\SH_{n,\ell}$ being a subrepresentation isomorphic to $\WH_{n,\ell-j-1}$.
\item If $j < \ell$ and $n=2$, then $\WH_{2,\ell} = \SH_{2,\ell}$.
\end{enumerate}
\end{thm}

Equation~\ref{eq:struct0} together with Lemma~\ref{lem:struct1} and Proposition~\ref{prop:struct3}, which are proved later in this section, imply Theorem~\ref{thm:Wnl}.

First we note that if $\ell - j - 1 < 0$, we have $\WH_{n,\ell-j-1} = \{ 0 \}$ and therefore $\SH_{n,\ell} = \{ 0 \}$. Hence: 

\begin{lemma} \label{lem:struct1}
Let $n, r \in \Npp$, $\ell \in \N$ and $0 \leq j \leq r-1$ such that $j \equiv n  + 2(\ell-1) \mod r$.
It holds that
\[
\SH_{n,\ell} \neq \{ 0 \} \Leftrightarrow j < \ell.
\]
Consequently, $\WH_{n,\ell} = \RH_{n,\ell}$, if $\ell \leq j < r - 1$.
\end{lemma}

We now state the following result, which we use to complete the proof of Theorem~\ref{thm:Wnl}.

\begin{prop}\label{prop:F_inj}
Let $n, r \in \Npp$ and $\ell \in \N$ such that $0 \leq \ell < r - 1$. The map $F \colon \VH_{n,\ell} \to \VH_{n,\ell+1}$ is injective.
\end{prop}

\begin{proof}
Let $0 \neq v \in \ker F \subset \VH_{n,\ell}$. Then $v$ can be represented as a linear combination of the following vectors in some of the projective modules in the direct sum decomposition of $\Vs_{r-1}^{\otimes n}$:
\begin{itemize}
\item the left-most socle vectors $u_0^\alpha \otimes w_i^S$ in a projective module $\Pmd_i$. Then, $u_0^\alpha \otimes w_0^R \in \VH_{n,\ell-(i+1+j)}=\VH_{n,\ell-(r-1)}$, where $j=r-2-i$. But, since $\ell-(r-1) < 0$, we have that $\VH_{n,\ell-(r-1)} = \{ 0 \}$, which is a contradiction.
\item the left-most vectors $u_0^\alpha \otimes w_j^L$ in a projective module $\Pmd_i$. Then $u_0^\alpha \otimes w_0^R \in \VH_{n,\ell-(2j+i+2)}= \VH_{n,\ell-(2r-2-i)}$, where $j=r-2-i$. Note that, it holds $0 \leq i \leq r-2$, so $r \leq 2r-2-i \leq 2r-2$. But, since $\ell < r - 1$, we have that $\ell-(2r-2-i) < 0$ and consequently $\VH_{n,\ell-(2r-2-i)} = \{ 0 \}$, which is a contradiction.
\item lowest weight vectors in some simple module $\Vmd_{r-1}$. Then, the highest weight vector of this simple module belongs in $\VH_{n,\ell-(r-1)}$. But $\VH_{n,\ell-(r-1)} = \{ 0 \}$, since $\ell < r-1$.
\end{itemize}
Therefore, $\ker F = \{ 0 \}$.
\end{proof}
 
Finally we show that the subspace $\SH_{n,\ell}$ is actually a subrepresentation and that $\WH_{n,\ell} \neq \{ 0 \}$ implies $\RH_{n,\ell} \neq \{ 0 \}$, if $n \geq 3$.

\begin{prop} \label{prop:struct3}
Let $n, r \in \Npp$ and $\ell \in \N$ such that $\ell < r$ and suppose that there exists $0 \leq j < \ell$ such that $j \equiv n + 2(\ell-1) \mod r$. Set $\ell':=\ell-1-j$.
Then we have the following.
\begin{enumerate}
\item The subspace $\SH_{n,\ell} \subset \WH_{n,\ell}$ is a $B_n$-submodule of $\WH_{n,\ell}$ isomorphic to $\WH_{n,\ell'}$;
\item for $n \geq 3$, $\RH_{n,\ell} \neq \{ 0 \}$.
\end{enumerate}
\end{prop}

\begin{proof}
We prove now the first statement. The map $F^{j+1} \colon \VH_{n,\ell'} \to \VH_{n,\ell}$ is injective by Proposition~\ref{prop:F_inj}. We show the following:
\[
\SH_{n,\ell} =  \Ima F^{j+1}\vert_{\WH_{n,\ell'}} \text{ as vector spaces}.
\]
The implication $\SH_{n,\ell} \subseteq \Ima F^{j+1}\vert_{\WH_{n,\ell'}}$ is clear by the definition of $\SH_{n,\ell}$. For the other direction, we solve for $\WH_{n,\ell'}$ the modular equation $j' \equiv n + 2(\ell'-1) \mod r$ with respect to $j'$:
\begin{align*}
j' \equiv n + 2(\ell'-1) \mod r
&\Leftrightarrow
j' \equiv -2(\ell-1) + j + 2(\ell - 1 - j - 1)  \mod r
\\
&\Leftrightarrow
 j' \equiv  r-2-j \mod r,
\end{align*}
where at the first equivalence we use the modular condition for $\WH_{n,\ell}$. Since $0 \leq j < \ell < r$, we have that $0 \leq r - 2 - j < r-1$ and hence $j' = r-2-j$. Moreover
\[
\ell < r
\Leftrightarrow \ell-1-j \leq r-2-j \Leftrightarrow \ell' \leq j'. 
\]
By Lemma~\ref{lem:struct1} it holds $\WH_{n,\ell'} = \RH_{n,\ell'}$ and therefore, by the structure of the projective modules $\Ps_j^{-\im \alpha}$, we have that $F^{j+1} \RH_{n,\ell'} \subseteq \SH_{n,\ell}$. Moreover, due to the injectivity of $F$ by Prop.~\ref{prop:F_inj} it holds $\dim \SH_{n,\ell} = \dim \WH_{n,\ell'}$. Finally, since the $B_n$-action commutes with the $D$-action, the space $\SH_{n,\ell}$ is isomorphic as a $B_n$-module to $\WH_{n,\ell'}$. 

Suppose now $n \geq 3$ for the second statement. Then $\dim \SH_{n,\ell} = \dim \WH_{n,\ell'} < \dim \WH_{n,\ell}$, so $\RH_{n,\ell}$ has dimension $\dim \WH_{n,\ell} - \dim \WH_{n,\ell'} > 0$.
\end{proof}

As mentioned before, Equation~\ref{eq:struct0}, Lemma~\ref{lem:struct1} and Proposition~\ref{prop:struct3} imply Theorem~\ref{thm:Wnl}.

\begin{rem}
A similar process has been applied to the quantum $\sltwo$ at roots of unity in \cite{Ito3}. The representations $\WH_{n,\ell}$ are isomorphic per definition to the representations $Y_{n,m}^N$ in \cite{Ito3} with the identifications $m=\ell$ and $N=r$. The author also defines by homological means a \textit{truncated} version of the Lawrence representation, which (by the results in \cite{Ito3}) is isomorphic to $\WH_{n,\ell}$. Note that by \cite[Lemma 2.5]{Ito1} we have that $\VH_{n,n(r-1)-m} \cong \VH_{n,m}$ (see Section~\ref{sec:dims} for the equality of dimensions).
\end{rem}

Finally, we prove that the direct sum $\SH_{n,\ell} \oplus \RH_{n,\ell}$ does not necessarily decompose as a direct sum of $B_n$-modules by providing a specific example.

\begin{lemma} \label{lem:RS_non_split}
Let $r=4$. The direct sum $\WH_{3,2} = \SH_{3,2} \oplus \RH_{3,2}$ is not a $B_3$-module decomposition.
\end{lemma}

\begin{proof}
Since $j \equiv n + 2(\ell-1) \equiv 1 \mod r$, by Proposition \ref{prop:struct3} we have $\RH_{3,2} \neq \{ 0 \}$ and $\SH_{3,2} \cong \WH_{3,0}$ as $B_3$-modules, where $\WH_{3,0}$ is the trivial $B_3$-representation.

There exists a basis $\{ w_{i,j} \mid 1 \leq i < j \leq n \}$ of $\WH_{n,2}$ such that the action of $B_n$ on $\WH_{n,2}$ \cite{JK}, where $\{i,i+1\} \cap \{j,k\} = \emptyset$ is given by
\begin{equation}\label{eq:LKB}
\begin{aligned}
\sigma_i w_{j,k} &= w_{j,k},\\
\sigma_i w_{i+1,j} &= s^{-1} \, w_{i,j},\\
\sigma_i w_{j,i+1} &= s^{-1} \, w_{j,i},\\
\sigma_i w_{i,j} &= s^{-1} \, w_{i+1,j} + (1-s^{-2}) \, w_{i,j} - s^{i-j-1}(1-s^{-2})q^2 \, w_{i,i+1},\\
\sigma_i w_{i,i+1} &= s^{-4}q^2 \, w_{i,i+1},\\
\sigma_i w_{j,i} &= s^{-1} \, w_{j,i+1} + (1-s^{-2}) \, w_{j,i} - s^{i-j-1}(1-s^{-2}) \, w_{i,i+1}.
\end{aligned}
\end{equation}
The above relations describe also the $B_n$-action for the 2-variable LKB representation, where $q$ and $s$ are considered as variables.

The matrices of the representation $\WH_{3,2}$ in the basis $\{ w_{1,2}, w_{1,3}, w_{2,3} \}$ are 
\[
\sigma_1 = \begin{pmatrix}
q^6 & q^3 - q & 0\\
0 & 1 - q^2 & q^5\\
0 & q^5 & 0
\end{pmatrix}
\andq
\sigma_2 = \begin{pmatrix}
1-q^2 & q^5 & 0 \\
q^5 & 0 & 0 \\
q^2 -1 & 0 & q^6
\end{pmatrix}
\]
The two matrices have eigenvalues $1$ and $q^6 = - \im$ with multiplicities $1$ and $2$ respectively. Since $\SH_{3,2} \cong \WH_{3,0}$, any vector of $\SH_{3,2}$ is an eigenvector for the eigenvalue $1$ for both matrices. In detail, since $\SH_{3,2} =  \Ima F^2 \vert_{\WH_{3,0}}$ as vector spaces and since $\WH_{3,0} = \spn \{ u_0^{\otimes 3} \}$ we have that $\SH_{3,2} = \spn \{ F^2 u_0^{\otimes 3} \}$, where
\[
F^2 u_0^{\otimes 3} = -(q+q^3) (q^6 \, w_{1,2} + q^3 w_{1,3} + w_{2,3}).
\]
Since $\sigma_i w_{i,i+1} = q^6 w_{i,i+1}$, for $i=1,2$, we have $w_{1,2}, w_{2,3} \not\in \SH_{n,\ell}$. 
The matrices of $\WH_{3,2}$ in the basis $\{ w_{1,2}, w_{2,3}, F^2 u_0^{\otimes 3}\}$ are given by
\[
\sigma_1 = \begin{pmatrix}
q^6 & -1 & 0\\
0 & q^6 & 0\\
0 & \frac{1}{q^3 + q^5} & 1
\end{pmatrix}
\andq
\sigma_2 = \begin{pmatrix}
q^6 & 0 & 0 \\
-1 & q^6 & 0 \\
\frac{1}{q^3 + q^5} & 0 & 1
\end{pmatrix}
\]
Suppose that $\WH_{3,2} = \RH_{3,2} \oplus \SH_{3,2}$ as $B_3$-modules. Then there exists a change of basis such that the $B_3$-action on $\RH_{3,2}$ is closed. That is, there exist $\lambda_1, \lambda_2 \neq 0$ such that $\{ w_{1,2} + \lambda_1 F^2 u_0^{\otimes 3}, w_{2,3} +  \lambda_2 F^2 u_0^{\otimes 3}\}$ is a basis of $\RH_{3,2}$ on which the $B_3$-action is closed. In this basis the matrices of $\WH_{3,2}$ are given by 
\begin{align*}
\sigma_1 &= \begin{pmatrix}
 q^6 & -1 & 0 \\
 0 & q^6 & 0 \\
 \left(1-q^6\right) \lambda _1 & \lambda _1 + \left(1- q^6\right) \lambda _2+\frac{1}{q^5+q^3} & 1 \\
\end{pmatrix}
\andq
\\
\sigma_2 &= \begin{pmatrix}
 q^6 & 0 & 0 \\
 -1 & q^6 & 0 \\
 \left(1-q^6\right) \lambda _1+\lambda _2+\frac{1}{q^5+q^3} & \left(1- q^6\right) \lambda _2 & 1 \\
\end{pmatrix}.
\end{align*}
Since the $B_3$-action decomposes the last row of both matrices equals $(0,0,1)$. But this implies $\lambda_1 = \lambda_2 = 0$, which is a contradiction.
\end{proof}

\section{Extensions of the Lawrence representations at roots of unity} \label{sec:NH}
Of great interest for this manuscript are the last two cases of Theorem~\ref{thm:Wnl}, where the representation $\WH_{n,\ell}$ contains the subspace $\SH_{n,\ell}$ consisting of the socle vectors of non-simple indecomposable projective modules $\Pmd_i$. In this section we extend the representation $\WH_{n,\ell}$ to a representation $\NH_{n,\ell}$ including also the dominant head vectors of $\Pmd_i$. As we will see, this extension is non trivial, in the sense that $\NH_{n,\ell}$ is non-simple and indecomposable.

\subsection{Definition and existence of the representations}
First, we define the representations using the decomposition of $\Vs_{r-1}^{\otimes n}$ to a direct sum of projective modules.

\begin{defn}
Let $n, r \in \Npp$ and $\ell \in \N$.
The \emph{dominant space} corresponding to the strong weight \mbox{$n(r-1) - 2 \ell$} is defined as:
\[
\NH_{n,\ell} := \ker(F E)^2 \cap \VH_{n,\ell}.
\]
\end{defn}

Since the action of $\Rs$ intertwines the $D$-action, $\NH_{n,\ell}$ is a representation of the braid group $B_n$. A direct consequence of the definition is that $\WH_{n,\ell} \subseteq \NH_{n,\ell}$ as a $B_n$-submodule. Set $d_{n,\ell} := \dim \WH_{n,\ell}$, for all $\ell \in \N$ as in Theorem~\ref{thm:Wnl_exist}.

\begin{thm}\label{thm:Nnl}
Let $n, r \in \Npp$ and $\ell \in \N$ such that $0 \leq \ell < r$. Let $0 \leq j \leq r-1$ be the solution of the modular equivalence $j \equiv n + 2(\ell-1) \mod r$. Then we have the following:
\begin{enumerate}
\item if $j \geq \ell$, $\NH_{n,\ell} = \WH_{n,\ell}$ as $B_n$-modules,
\item if $j < \ell$, 
$\NH_{n,\ell} = \WH_{n,\ell} \oplus \HH_{n,\ell}$ as vector spaces, where
\[
\HH_{n,\ell} := \widetilde{x}_{\alpha_i}^{-1}(\SH_{n,\ell})
\] 
with the bijection $\widetilde{x}_{\alpha_i}$ defined as in \eqref{eq:proj_bij}.
Moreover, $\dim \HH_{n,\ell} = \dim \SH_{n,\ell} = \dim \WH_{n,\ell'}$, where $\ell'= \ell - 1 - j$. Finally, $\NH_{n,\ell}$ as a $B_n$-representation is given by a homomorphism:
\[
\rho_{n,\ell}^{\NH} \colon B_n \to \mathrm{GL} \left( d_{n, \ell} + d_{n,\ell'} , \C \right).
\]
\end{enumerate}
\end{thm}

\begin{proof}
Let $j \geq \ell$ and suppose there exists a vector $w \in \NH_{n,\ell}$ such that $w \not \in \WH_{n,\ell}$. By the definition of $\NH_{n,\ell}$ the vector $w$ is represented by a dominant head vector of an indecomposable projective module in the direct sum decomposition of $\Vs_{r-1}^{\otimes n}$. So, $F E \, w \in \SH_{n,\ell} \subset \WH_{n,\ell}$. But by Theorem~\ref{thm:Wnl} we have that $\WH_{n,\ell}=\RH_{n,\ell}$, so $FE w = 0$, which is a contradiction.

Let now $j < \ell$. By Theorem~\ref{thm:Wnl} there exists a subspace $\SH_{n,\ell} \subset \WH_{n,\ell}$ spanned by the socle highest weight vectors $\left( u_0^\alpha \otimes w_0^S \right)_m$, where $m=1, \ldots, \dim \SH_{n,\ell}$, which belong in $\dim \SH_{n, \ell}$ copies of the indecomposable modules $\Pmd_i$ in the direct sum decomposition of $\Vs_{r-1}^{\otimes n}$. Recall that there exists a bijection $\widetilde{x}_{\alpha,i} \colon \head(\Pmd_i) \to \soc(\Pmd_i)$. In detail, for every vector $\left( u_0^\alpha \otimes w_0^S \right)_m$ there is a corresponding dominant head vector $\left( u_0^\alpha \otimes w_0^H \right)_m \in \Pmd_i$ such that $\widetilde{x}_{\alpha,i} \left( u_0^\alpha \otimes w_0^H \right)_m = \left( u_0^\alpha \otimes w_0^S \right)_m$. Note that
\[
\left( u_0^\alpha \otimes w_0^H \right)_m \in \NH_{n,\ell} \andq \left( u_0^\alpha \otimes w_0^H \right)_m \not\in \WH_{n,\ell}, \quad \forall m \in \{1, \ldots, \dim \SH_{n,\ell}\}.
\]
We define
\[
\HH_{n, \ell} := \widetilde{x}_{\alpha_i}^{-1}(\SH_{n,\ell}).
\]
Since $\widetilde{x}_{\alpha,i}$ is a bijection, we have that $\NH_{n,\ell} = \WH_{n,\ell}  \oplus \HH_{n,\ell}$. Finally, by the definition of $\HH_{n, \ell}$ it is immediate that $\dim \HH_{n, \ell} =  \dim \SH_{n, \ell}$.
\end{proof}

\begin{figure}[ht]
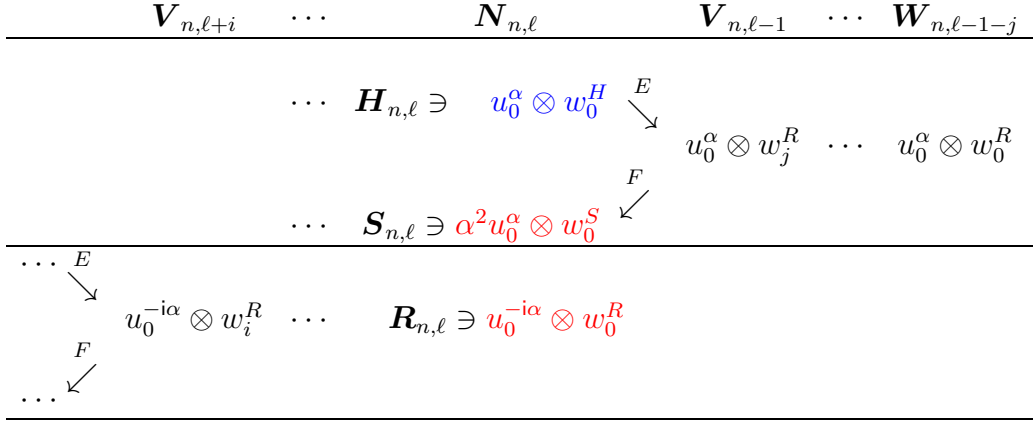

\[
\begin{matrix}
& \VH_{n,\ell+i} & \cdots &\NH_{n,\ell} & \VH_{n,\ell-1} & \cdots & \WH_{n,\ell-1-j}
\\\hline
\\
& & \cdots &\HH_{n,\ell} \ni \textcolor{blue}{\phantom{\alpha^2} u_0^\alpha \otimes w_0^H} \;
\stackrel{E}{\raisebox{-0.5\height}{$\searrow$}}&
\\
& & & & u_0^\alpha \otimes w_j^R & \cdots & u_0^\alpha \otimes w_0^R
\\
& & \cdots &
\SH_{n,\ell} \ni \textcolor{red}{\alpha^2 u_0^\alpha \otimes w_0^S} \;
\stackrel{F}{\raisebox{\height}{$\swarrow$}}&
\\\hline
\cdots \stackrel{E}{\raisebox{-\height}{$\searrow$}}&
\\
&u_0^{- \im \alpha} \otimes w_i^R & \cdots & \RH_{n,\ell} \ni \textcolor{red}{u_0^{- \im \alpha} \otimes w_0^R} \;
\\
\cdots \stackrel{F}{\raisebox{\height}{$\swarrow$}}&\\\hline
\end{matrix}
\]
\caption{The vectors in $\Pmd_i$ and $\Ps_j^{- \im \alpha}$ spanning $\NH_{n,\ell}$ for $n \geq 3$. }\label{fig:rep_vecs}
\end{figure}

We visualize in Figure~\ref{fig:rep_vecs} the vectors in the direct sum decomposition of $\Vs_{r-1}^{\otimes n}$ spanning the space $\NH_{n,\ell}$. There exist $\dim \SH_{n,\ell} = \dim \WH_{n, \ell'}$ copies of the projective module $\Pmd_i$ representing vectors of $\NH_{n,\ell}$. These contain the dominant but not highest weight vectors comprising $\HH_{n,\ell}$ (colored by blue) and the highest weight vectors comprising $\SH_{n,\ell}$ (colored by red). The $\dim \RH_{n,\ell}$ copies of $\Ps_j^{- \im \alpha}$ contain the highest weight vectors of $\RH_{n,\ell} \subset \WH_{n, \ell}$. Recall that if $n=2$, then $\RH_{n,\ell} = \{ 0 \}$.

In Proposition~\ref{prop:struct3} we showed that $\SH_{n,\ell} \cong \WH_{n,\ell'}$, where $\ell' = j - 1 - \ell$, and by Theorem~\ref{thm:Nnl} we have that $\dim \HH_{n,\ell} = \dim \WH_{n,\ell'}$. But moreover, the representation $\NH_{n,\ell}$ includes $\WH_{n,\ell'}$
also as a quotient as proved in the following.

\begin{prop} \label{prop:exseq1}
Let $n, r \in \Npp$ and $\ell \in \N$ such that $0 \leq \ell < r$. Suppose that there exists $0 \leq j < \ell$ such that $j \equiv n + 2(\ell-1) \mod r$
and set $\ell' := \ell-1-j$. There exists a short exact sequence of $B_n$-modules
\[
\begin{tikzcd}
0 \arrow[r] & \WH_{n,\ell} \arrow[r, hook] & \NH_{n,\ell} \arrow[r] & \WH_{n,\ell'} \arrow[r] & 0.
\end{tikzcd}
\]
\end{prop}

\begin{proof}
Since $j < \ell$, we have by Theorem~\ref{thm:Nnl} that $\dim \HH_{n,\ell} = \dim \SH_{n,\ell} = d_{n,\ell'}$, where $d_{n,\ell'} = \dim \WH_{n,\ell'}$.
As a vector space $\NH_{n,\ell}/\WH_{n,\ell}$ is spanned by the images of the vectors of $\HH_{n,\ell}$ under the quotient map. Let $\{ h_k \mid k=1, \ldots, d_{n,\ell'} \}$ be a basis of $\HH_{n,\ell}$. For any $p = 1, \ldots, n-1$ the action of $B_n$ on the quotient $\NH_{n,\ell}/\WH_{n,\ell}$ is induced by the action of $B_n$ on $\NH_{n,\ell}$, that is:
\[
\sigma_p (h_k) = \sum_{k'=1}^{d_{n,\ell'}} \lambda_{p, k'}^{(k)} \, h_{k'} \mod \WH_{n,\ell}, \quad \text{ for some } \lambda_{p, k'}^{(k)} \in \C.
\]
Since $j < \ell$, by Proposition~\ref{prop:struct3} the subspace $\SH_{n,\ell}$ is a $B_n$-subrepresentation isomorphic to $\WH_{n,\ell'}$. Further, by Theorem~\ref{thm:Nnl}
there exists
a basis $\{ y_k \mid k=1, \ldots, d_{n,\ell'} \}$ of $\SH_{n,\ell}$ such that $y_k = x_{\alpha, i} (h_k) = \alpha^2 FE (h_k)$. We write the $B_n$-action on $\SH_{n,\ell} \cong \WH_{n,\ell'}$ as:
\[
\sigma_p (y_k) = \sum_{k'=1}^{d_{n,\ell'}} \mu_{p, k'}^{(k)} \, y_{k'}, \text{ for some } \mu_{p, k'}^{(k)} \in \C.
\]
Since the $B_n$-action commutes with the $D$-action we have for every $1 \leq k \leq d$:
\[
\sigma_p (y_k) = \sigma_p (\alpha^2 FE \, h_{k}) = \alpha^2 FE \, \sigma_p (h_{k})
= \alpha^2 FE \left(  \sum_{k'=1}^{d_{n,\ell'}} \lambda_{p, k'}^{(k)} \, h_{k'} + w \right),
\]
for some $w \in \WH_{n,\ell}$. Since $E \vert_{\WH_{n,\ell}} =0$ we have consequently for every $1 \leq k \leq d$:
\[
\sigma_p (y_k) = \alpha^2 FE \sum_{k'=1}^{d_{n,\ell'}} \lambda_{p, k'}^{(k)} \, h_{k'}
= \sum_{k'=1}^{d_{n,\ell'}} \lambda_{p, k'}^{(k)}  \alpha^2 FE \, h_{k'} = \sum_{k'=1}^{d_{n,\ell'}} \lambda_{p, k'}^{(k)}  \, y_k.
\]
Finally, $\lambda_{p, k'}^{(k)} = \mu_{p, k'}^{(k)}$, for every $p,k'$ and $k$, which proves the statement.
\end{proof}

\subsection{Faithfulness on \texorpdfstring{$Z(B_n)$}{Z(Bn)}}
The next question arising is whether the exact sequence of Proposition~\ref{prop:exseq1} is non-split, that is, whether $\NH_{n,\ell}$ is decomposable or not. In this section we show that the representation $\NH_{n, \ell}$ is faithful on $Z(B_n)$, the center of $B_n$. This implies that the sequence does not split and, therefore, the representation $\NH_{n,\ell}$ is a non trivial extension of the representation $\WH_{n, \ell}$.

For all $2 \leq i \leq n$ we define the braid $\delta_i \in B_n$ by:
\[
\delta_i := \sigma_1 \sigma_2 \ldots \sigma_{i-1}.
\]
The center of $B_n$ is generated by the braid $\theta_n := \Delta_n^2$ \cite{KaTu}, where:
\[
\Delta_n := \delta_n \delta_{n-1} \ldots \delta_2.
\]
It is easy to prove using the braid relations that
\[
\theta_n = \Delta_n^2 = \delta_n^n.
\]
Both braids $\Delta_n$ and $\delta_n$ have played an important role to the study of the word problem in $B_n$, both of them being Garside elements for two different Garside structures on $B_n$. The braid $\Delta_n$ gives rise to the usual Garside structure \cite{Gar} and the braid $\delta_n$ to the dual Garside structure \cite{BKL} (see also \cite{Ito2}). Moreover, the LKB representation (and hence also the isomorphic quantum representation $\Wb_{n,2}$) detects the Garside structure \cite{Kr2} and the whole family of the Lawrence representations (isomorphic to $\Wb_{n,\ell}$) detect the dual Garside structure as proved in \cite{Ito2}. Both results can be used to alternatively prove the faithfulness of the respective representations.

\bigbreak
To prove the faithfulness of $\NH_{n,\ell}$ on $Z(B_n)$ it suffices to prove that the action of the braid $\theta_n$ has infinite order. To do this we represent the action of $\theta_n$ on $\Vs_{r-1}^{\otimes n}$ by the analogous map in the category $D$-mod, which involves the twist operator $\theta_{\Vs_{r-1}^{\otimes n}}$. Using then the decomposition of $\Vs_{r-1}^{\otimes n}$ into indecomposable projective modules we compute the action of $\theta_n$ by restricting it on $\NH_{n,\ell}$, whose vectors belong to non-simple projective modules (see Theorem~\ref{thm:Nnl} and Figure~\ref{fig:rep_vecs}). The infinite order of the action of $\theta_n$ on $\NH_{n,\ell}$ is a consequence of the infinite order of the action of the twist on the non-simple indecomposable projective modules \eqref{eq:twist_proj}. Finally, we give an explicit formula for the action of $\theta_n$, as follows.

\begin{thm} \label{thm:faithful_center}
Let $n, r \in \Npp$ and $\ell \in \N$. Let $0 \leq j < r$ be the solution of the modular equivalence $j \equiv n + 2(\ell-1) \mod r$ and let $\ell' = \ell - 1 - j$.
The action of the full twist $\theta_n$ on $\NH_{n,\ell}$ is given by:
\[
\theta_n = q^{2 \ell (n + \ell -1)} \left( \Id + s^n q^{1 - \ell - \ell'} (\ell - \ell') \frac{\qnum{1}^2}{\qnum{\ell - \ell'}} \, F E \right).
\]
\end{thm}

\begin{proof}
First we adapt \eqref{eq:twist_proj} to any indecomposable projective module $\Pmd_i$. Note that (see Appendix~\ref{app:double_braiding} for the proof):
\begin{align}
c_{\Ps_i, \Vmd_0} \circ c_{\Vmd_0, \Ps_i} &= (-1)^{m i} \, \Id_{\Pmd_i}, \label{eq:double_braiding}\\
\theta_{\Vmd_0} &= q^{\frac{(mr)^2}{2} + mr - m r^2} \,\Id_{\Vmd_0}, \label{eq:twist_cmr}
\end{align}
where $0 \leq i < r-1$ and where we write $\alpha = q^{mr/2}$, for some $m \in \Z$.
Using \eqref{eq:double_braiding}, \eqref{eq:twist_cmr} and the naturality of the twist, it holds that (see Appendix~\ref{app:twist_Pmdi} for the proof):
\begin{equation} \label{eq:twist_Pmdi}
\theta_{\Pmd_i} = (-1)^{mr+i} q^{\frac{(mr+i)^2}{2}+(mr+i)} \left( I_{\alpha,i} - (r - i - 1) \frac{\qnum{1}^2}{\qnum{i+1}} x_{\alpha,i} \right).
\end{equation}

Now, by the naturality of the twist, \eqref{eq:Bn_action_V}, \eqref{eq:sigma_i} and the equality $\theta_{\Vs_{r-1}} = q^{-\frac{(r-1)^2}{2}} \,\Id_{\Vs_{r-1}}$, the braid $\theta_n$ acts on $\Vs_{r-1}^{\otimes n}$ by the map:
\[
\Theta_n := q^{-n(n-2)\frac{(r-1)^2}{2}} \theta_{\Vs_{r-1}^{\otimes n}}.
\]
By the direct sum decomposition of $\Vs_{r-1}^{\otimes n}$ into projective modules we can write:
\[
\theta_{\Vs_{r-1}^{\otimes n}} = \theta_{\bigoplus_{\alpha \in I} \bigoplus_{m \in J} \Pmd_m}
=  \bigoplus_{\alpha \in I} \bigoplus_{m \in J} \theta_{ \Pmd_m },
\]
where $I \subset \Z$ and $J \subset \{ 0, \ldots, r-1 \}$. Restricting the map $\theta_{\Vs_{r-1}^{\otimes n}}$ on $\NH_{n,\ell}$ we get by Theorem~\ref{thm:Nnl}:
\[
\theta_{\Vs_{r-1}^{\otimes n}} \vert_{\NH_{n,\ell}} = \theta_{\Pmd_i} \vert_{\NH_{n,\ell}} \oplus \theta_{\Ps_j^{- \im \alpha}} \vert_{\NH_{n,\ell}} ,
\]
where $0 \leq i < r-2$ and where $j = r - 2 - i$.
By \eqref{eq:twist_cmr} we have that:
\[
\theta_{\Pmd_i}  = A_{\alpha,i} \left( I_{\alpha,i} - (r - i - 1) \frac{\qnum{1}^2}{\qnum{i+1}} x_{\alpha,i} \right).
\]
where
\[
A_{\alpha,i} =  (-1)^{\xi} q^{\frac{\xi^2}{2} + \xi}, \quad \text{ with } \xi := mr + i.
\]
Since $(m-1)r+j = \xi - 2(i+1)$ we get that:
\begin{align*}
A_{- \im \alpha, j} &= (-1)^{\xi - 2(i+1)} q^{\frac{(\xi - 2(i+1))^2}{2}+\xi - 2(i+1)} 
= (-1)^{\xi} q^{\frac{\xi^2}{2} - 2\xi(i+1) + 2(i+1)^2 + \xi - 2(i+1)} 
\\
&= (-1)^{\xi} q^{\frac{\xi^2}{2} + \xi} q^{- 2\xi(i+1) + 2(i+1)^2 - 2(i+1)}
= (-1)^{\xi} q^{\frac{\xi^2}{2} + \xi} q^{- 2i(i+1) + 2(i+1)^2 - 2(i+1)}
\\
&= (-1)^{\xi} q^{\frac{\xi^2}{2} + \xi}
= A_{\alpha,i},
\end{align*}
where at the fourth equality we use that $q^{2 \xi} = q^{2i}$.
Moreover, since $\NH_{n,\ell} \cap \left( \Ps_j^{- \im \alpha} \right)$ contains only highest weight vectors, the nilpotent map $x_{-\im \alpha,j}$ acts by $0$ on them. Hence:
\[
\theta_{\Ps_j^{- \im \alpha}} \vert_{\NH_{n,\ell}} = A_{\alpha,i} \, I_{-\im \alpha,j}.
\]
Putting everything together:
\[
\Theta_n \vert_{\NH_{n,\ell}} = q^{-n(n-2)\frac{(r-1)^2}{2}} A_{\alpha,i} \left( \Id - (-1)^m (\ell - \ell') \frac{\qnum{1}^2}{\qnum{\ell - \ell'}}  F E \right),
\]
where we also use \eqref{eq:xmi} and that $i+1=r-(j+1)=r-(\ell - \ell')$.

Now it remains to compute the coefficient $q^{-n(n-2)\frac{(r-1)^2}{2}} A_{\alpha,i}$. Since we work on $\NH_{n,\ell}$ we have that $q^{\frac{n(r-1)-2\ell}{2}}=q^{\frac{\xi}{2}}$. Using also that $q^{2r}=1$ we get
\begin{align*}
A_{\alpha,i} &= (-1)^{n(r-1) - 2 \ell} q^{\frac{(n(r-1) - 2 \ell)^2}{2}} q^{n(r-1) - 2 \ell} 
\\
&= q^{-r(n(r-1) - 2 \ell)} q^{\frac{n^2 (r-1)^2}{2} - 2n(r-1) \ell + 2 \ell^2} q^{n(r-1) - 2 \ell} 
\\
&= q^{\frac{n^2 (r-1)^2}{2} + 2n \ell + 2 \ell^2} q^{n(r-1) - 2 \ell} q^{-rn(r-1)}.
\end{align*}
Hence:
\begin{align*}
q^{-n(n-2)\frac{(r-1)^2}{2}} A_{\alpha,i} &= q^{n(r-1)^2 + 2n \ell + 2 \ell^2 + n(r-1) - 2 \ell -rn(r-1)}
\\
&= q^{n(r-1)r + 2n \ell + 2 \ell^2 - 2 \ell -rn(r-1)}
= q^{2 \ell (n + \ell -1)}.
\end{align*}
Finally for the coefficient of $F E$ we have that:
\[
(-1)^{m+1} = q^{n(r-1)-2 \ell - i + r} = q^{n(r-1) -2 \ell + 2 + j} = s^n q^{-2(\ell - 1) + j} = s^n q^{1 - \ell - \ell'},
\]
which concludes the proof for the formula of the action of $\theta_n$. 
\end{proof}

It is immediate due to Theorem~\ref{thm:faithful_center} that the representations $\WH_{n,\ell}$ are not faithful. Since the full twist $\theta_n$ acts on $\WH_{n,\ell}$ by the scalar $q^{2 \ell (n + \ell -1)}$, we have consequently that $\theta_n^r \in \ker \rho_{n,\ell}^{\WH}$. Further, the formula of the action of $\theta_n$ depends only on $n, r$ and $\ell$, since $\ell'$ already depends on them.

\begin{rem} \label{rem:non_trivial_exts}
Note that the non trivial extensions $\NH_{n,\ell}$ of $\WH_{n,\ell}$ depend actually on the triple $(n, \ell, \ell')$ where $\ell' = \ell -j -1$ with $0 \leq j < \ell$ such that $j \equiv n + 2(\ell-1) \mod r$. Note that
\[
\ell \equiv 1 - n - \ell' \mod r
\Leftrightarrow
j \equiv n + 2(\ell - 1) \mod r.
\]
Therefore, we denote these non trivial extensions from now on by $\NH_{n,\ell,\ell'}$ as in Theorem~\ref{thm:thm1}. 
\[
\NH_{n, \ell, \ell'} := \ker(FE)^2 \cap \VH_{n,\ell} \; \text{ if } \exists \ell' \in \{0, \ldots, \ell-1\} \; \text{ s.t. } \; \ell' \equiv 1 - n - \ell \mod r.
\]
If such a $\ell'$ does not exist, we use the notation $\NH_{n,\ell}$, since in that case $\NH_{n,\ell} = \WH_{n,\ell}$.
\end{rem}

\bigbreak
Using Theorem~\ref{thm:faithful_center} we now prove the main Theorem of this manuscript.

\begin{proof}[Proof of Theorem~\ref{thm:thm1}]
The short exact sequence of the statement exists due to Theorem~\ref{thm:Nnl} and Proposition~\ref{prop:exseq1}. The modular condition of Theorem~\ref{thm:thm1} is due to Remark~\ref{rem:non_trivial_exts} equivalent to the one of Proposition~\ref{prop:exseq1}.
By Theorem~\ref{thm:faithful_center} the action of the full twist $\theta_n$ on $\NH_{n,\ell,\ell'}$ is given by
\[
\theta_n = q^{2 \ell (n + \ell -1)} \left( \Id + s^n q^{1 - \ell - \ell'} (\ell - \ell') \frac{\qnum{1}^2}{\qnum{\ell - \ell'}} \, F E \right).
\]
Since $\ell' < \ell$, and therefore $\NH_{n,\ell,\ell'} \supsetneq \WH_{n,\ell}$, the map $FE$ is a non-zero map. Now, the coefficient of $F E$ is clearly not a root of unity for $\ell' < \ell - 1$. For $\ell' = \ell + 1$ we have that 
$ (\ell - \ell') \frac{\qnum{1}^2}{\qnum{\ell - \ell'}} = \qnum{1}$, which is also not a root of unity for any $r \geq 2$. Hence, $\theta_n$ has infinite order and the action of $Z(B_n)$ is faithful. Therefore, for every $w_1 \in \HH_{n,\ell}$ and any $w_2 \in \WH_{n,\ell}$ we have that:
\begin{align*}
\theta_n w_1 &= q^{2 \ell (n + \ell -1)} \left( w_1 + s^n q^{1 - \ell - \ell'} (\ell - \ell') \frac{\qnum{1}^2}{\qnum{\ell - \ell'}} \, F E \, w_1 \right),
\\
\theta_n w_2 &= q^{2 \ell (n + \ell -1)}  w_2,
\end{align*}
with $F E \, w_1 \in \SH_{n,\ell}$. Hence, the exact sequence of the statement does not split.
\end{proof}

\begin{rem}
Note further that for every $n \in \Npp$ and $\ell' \in \N$ there exist $r \in \Npp$ and $\ell \in \N$ with $\ell' < \ell < r$ such that the modular condition $\ell' \equiv 1 - n - \ell \mod r$ is satisfied.  For example, setting $r = n + 2 \ell'$ we get $\ell = \ell' + 1$, which satisfies the modular condition.
\end{rem}

\section{Extensions of the LKB representation at roots of unity} \label{sec:examples}
In this section we describe explicitly the action of $B_n$ for the representations $\NH_{n,2,0}$ and $\NH_{n,2,1}$ that contain the LKB representation at roots of unity. In detail $\NH_{n,2,0}$ extends the trivial representation by LKB and $\NH_{n,2,1}$ extends the (reduced) specialized Burau representation by the specialized LKB.

\subsection{The case \texorpdfstring{$\ell=1$}{l=1}}
Let us warm up by studying the representations $\NH_{n,1}$ and in particular $\NH_{n,1,0}$. 
By Theorem~\ref{thm:Nnl} there exists for $\ell =1$ a non trivial extension of $\WH_{n,1}$ if and only if $n \equiv 0 \mod r$. Since $\dim \VH_{n,1} = n = (n-1) + 1 =  \dim  \WH_{n,1} +  \dim \WH_{n,0}$, we have consequently by Theorem~\ref{thm:Nnl} that
\begin{align*}
\NH_{n,1,0} = \VH_{n,1}& &\text{ if } n \equiv 0 \mod r,
\\
\NH_{n,1} = \WH_{n,1}& &\text{ if } n \not\equiv 0 \mod r.
\end{align*}
In \cite{JK} the representation $\Wb_{n,1}$ (resp. $\Vb_{n,1}$) is proved to be isomorphic to the reduced (resp. unreduced) Burau representation. In detail, the vectors $c_i$, $1 \leq i \leq n$ \eqref{eq:ci} span the space $\Vb_{n,1}$. Changing now the basis by setting $\widehat{c}_i := s^i\, c_i$ for $1 \leq i \leq n$ and by computing the $B_n$-action we obtain:
\[
\begin{aligned}
\sigma_i \, \widehat{c}_j &= \widehat{c}_j \qquad \text { for } j \neq i, i+1,\\
\sigma_i \, \widehat{c}_i &= \ts \, \widehat{c}_{i+1} + (1 - \ts)\, \widehat{c}_i,\\
\sigma_i \, \widehat{c}_{i+1} &= \widehat{c}_i.
\end{aligned}
\]
This basis gives the well-known matrices of the unreduced Burau representation \cite{Bur} and we have that $\Vb_{n,1}$ is isomorphic to the Burau representation by setting $\ts = s^{-2}$ \cite{JK}. We denote by $\widetilde{\rho}_\text{Bur}(\ts) \colon B_n \to GL_n(\Z[\ts^\pm])$ the unreduced Burau representation and by $\rho_\text{Bur}(\ts) \colon B_n \to GL_{n-1}(\Z[\ts^\pm])$ its reduced version.

Passing now to the representation $\VH_{n,1}$ we recover the Burau representation, specialized at $\ts = s^{-2} = q^2$.  It is well known, that $\ker \widetilde{\rho}_\text{Bur}(\ts) = \ker \rho_\text{Bur}(\ts)$. If we specialize $\ts$ to a complex number then the two representations are equivalent when $1 + \ts + \ldots + \ts^{n-1} \neq 0$ \cite{Bir}. But since $\ts = q^2$ and $n \equiv 0 \mod r$ for our case, the sum is a multiple of the sum of all $r$-th roots of unity, which is 0 and we have that $\ker \widetilde{\rho}_\text{Bur}(q^2) \subseteq \ker \rho_\text{Bur}(q^2)$. Moreover, by Theorem~\ref{thm:faithful_center} it holds actually that $\ker \widetilde{\rho}_\text{Bur}(q^2) \subsetneq \ker \rho_\text{Bur}(q^2)$.

\subsection{The case \texorpdfstring{$\ell=2$}{l=2}}
We continue now to studying the non trivial extensions of $\WH_{n,2}$.
In this case we get two families of representations, the representations $\NH_{n,2,0}$ extending the the trivial representation by the LKB (specialized at roots of unity) representation and the representations $\NH_{n,2,1}$ extending the specialized Burau representation by the specialized LKB representation. By Theorem~\ref{thm:Nnl} we have that
\begin{align*}
\NH_{n,2,0} = \WH_{n,0} \oplus \WH_{n,2} \quad \text{(as vector spaces)} \quad  \text{ if } \quad n \equiv -1 \mod r,
\\
\NH_{n,2,1} = \WH_{n,1} \oplus \WH_{n,2} \quad \text{(as vector spaces)} \quad  \text{ if } \quad n \equiv -2 \mod r.
\end{align*}
In order to explicitly calculate the $B_n$-action on the spaces $\NH_{n,2,0}$ and $\NH_{n,2,1}$ we need first to find a suitable basis. For their subspace $\WH_{n,2}$ we can describe a basis using the isomorphism $\AH_{n,\ell} \stackrel{\Phi}{\cong} \WH_{n,\ell}$. We fix the following notation for the basis of the space $\VH_{n,2} = \AH_{n,2} \oplus \BH_{n,2}$ according to \eqref{eq:V_nl_basis}:
\begin{equation*}
\begin{aligned}
\AH_{n,2} &= \spn \{ a_{i,j} := u_0^{\otimes i -1} \otimes u_1 \otimes u_0^{\otimes j - i -1} \otimes u_1 \otimes u_0^{\otimes n - j} \mid 1 \leq i < j \leq n \},
\\
\BH_{n,2} &= \spn \{ b_i := u_0^{\otimes i -1} \otimes u_2 \otimes u_0^{\otimes n - i} \mid 1 \leq i \leq n \}.
\end{aligned}
\end{equation*}
Hence, the space $\WH_{n,2}$ is spanned by the vectors of the form
\begin{equation} \label{eq:W_n2_basis}
w_{i,j} := \Phi(a_{i,j}) = a_{i,j} - s^{j-i} q^{-2} \, b_j - s^{i-j} \, b_i \quad \text{ for } \quad 1 \leq i < j \leq n.
\end{equation}
Recall that $s = q^{r-1}$ \eqref{eq:s}. Moreover, the action of $B_n$ on $\WH_{n,2}$ \cite{JK}, where $\{i,i+1\} \cap \{j,k\} = \emptyset$ is given by \eqref{eq:LKB},
which also describes the $B_n$-action for the 2-variable LKB representation, where $q$ and $s$ are considered as variables.

Let now $\sum_{i=1}^n \lambda_i b_i$ be an element of $\BH_{n,2}$. In order to find the coefficients $\lambda_i$, so that it belongs to $\NH_{n,2,0}$ or $\NH_{n,2,1}$, we compute the action of $E F E$ on the basis vectors $b_i$. For the action of $E$ on $b_i$, it holds that $E \, b_i = s^{n-i} \, c_i$, where $c_i = u_0^{\otimes i -1} \otimes u_1 \otimes u_0^{\otimes n-i}$ \eqref{eq:ci}. Now, we compute $F \, c_i$:
\begin{equation} \label{eq:Fci}
F  \,  c_i =  s^{-(i-1)} q^2 \sum_{j=1}^{n-i} s^{-j} \, a_{i,i+j} 
+  \sum_{j=1}^{i-1} s^{-(j-1)}\, a_{j,i} + \qint{2}^2 s^{-(i-1)} \, b_i.
\end{equation}
For the proof, see Appendix~\ref{app:Fci}. We now write $F c_i$ with respect to the basis of the decomposition $\VH_{n,2} \cong \WH_{n,2} \oplus \BH_{n,2}$. By \eqref{eq:W_n2_basis} we have that:
\begin{equation}\label{eq:aij_wij}
w_{i,i+j} = a_{i,i+j} - s^{j} q^{-2} \, b_{i+j} - s^{-j} \, b_i
\andq
w_{j,i} = a_{j,i} - s^{i-j} q^{-2} \, b_i - s^{j-i} \, b_j.
\end{equation}
Substituting them into \eqref{eq:Fci} it follows that:
\begin{equation} \label{eq:Fci2}
F\, c_i = s^{-(i-1)} q^2 \sum_{j=1}^{n-i} s^{-j} \, w_{i,i+j} + \sum_{j=1}^{i-1} s^{-(j-1)} \, w_{j,i} + s^{-(i-1)} \left[ \sum_{\substack{j=1 \\ j \neq i}}^{n} b_{j} + \beta_i \, b_i \right],
\end{equation}
where:
\begin{equation} \label{eq:beta_i}
\beta_i 
= 1 + \frac{q^{-2i}(1-q^{2n+2})}{1-q^2}
= \begin{cases}
q^{-2i } + 1 = s^{2i} + 1  & \text{ if } n \equiv -1  \mod r,
\\
1 & \text{ if } n \equiv -2  \mod r.
\end{cases}
\end{equation}
Since $E w_{i,j} = 0$, for all $1 \leq i < j \leq n$, we conclude by \eqref{eq:Fci2} that:
\begin{equation}\label{eq:lin_sys}
\begin{aligned}
E F E \, &\left( \sum_{i=1}^n \lambda_i \, b_i \right) = 0
\Leftrightarrow \sum_{j=1}^n \left[ \sum_{\substack{i=1 \\ i \neq j}}^{n} \lambda_i s^{-2i-j}  + \lambda_j s^{-3j} \beta_j \right] \, c_j  = 0
\\
&\Leftrightarrow  \sum_{\substack{i=1 \\ i \neq j}}^{n} \lambda_i s^{-2i-j} + \lambda_j s^{-3j} \beta_j  = 0 \quad \text{ for all } j=1,\ldots,n.
\end{aligned}
\end{equation}
The solutions to the coefficients $\lambda_i$ depend on the value of $\beta_j$ \eqref{eq:beta_i}, which differs depending on the modular condition between $n$, $\ell$ and $r$ of Theorem~\ref{thm:Nnl}. In the next sections, we study the two cases separately.

\subsection{The case \texorpdfstring{$\ell=2$}{l=2} and \texorpdfstring{$n \equiv -1 \mod r$}{n=-1 mod r}} \label{sec:N20}

We solve now the linear system \eqref{eq:lin_sys} for the representation $\NH_{n,2,0}$. Since $\beta_j = q^{-2j} + 1 = s^{2j}+1$, we get that:
\[
\sum_{i=1}^{n} \lambda_i s^{-2i-j} + \lambda_j s^{-j} = 0 \quad \text{ for all } j=1,\ldots,n.
\]
It is easy to see that $\lambda_1 = \ldots = \lambda_n = 1$ is a solution for the system.
So, by Theorem~\ref{thm:Nnl}, we have that for $n \equiv -1 \mod r$:
\[
\NH_{n,2,0} = \spn \{ b := b_1 + \ldots + b_n \} \oplus \WH_{n,2} , \quad \text{(as vector spaces)}.
\]
Using the $R$-matrix \eqref{eq:Rmatrix_V}, the definition of the $B_n$-action \eqref{eq:Bn_action_V} and \eqref{eq:aij_wij}, we compute the $B_n$-action on the basis vectors of $\BH_{n,2}$:
\begin{equation}\label{eq:action_BH_4}
\begin{aligned}
\sigma_i b_j &= b_j \qquad \text{ for } j \neq i, i+1,\\
\sigma_i b_i &= s^{-3} (1 - q^2) \, w_{i,i+1} + (1-q^2)\, b_i + b_{i+1},\\
\sigma_i b_{i+1} &= q^2 \, b_i.
\end{aligned}
\end{equation}
Finally it holds that
\begin{equation}\label{eq:action_vec_b}
\sigma_i b = b + s^{-3} (1-q^2) \, w_{i,i+1}.
\end{equation}
By Theorem~\ref{thm:faithful_center} we know that $\NH_{n,2,0}$ is faithful on $Z(B_n)$. In detail:
\begin{equation} \label{eq:twist_LKB_triv_ext}
\theta_n b = b - \frac{2 \im \, s^n q^{-1}}{\tan(\pi / r )} FE \, b
\andq
\theta_n w_{i,j} = w_{i,j}.
\end{equation}
The action of $\theta_n$ on the vectors $w_{i,j}$ can be proved using alternatively a result of Krammer about the action of $\Delta_n$ on the LKB represetantion \cite[Lemma 3.2]{Kr2} and using the isomorphism between the space $\Wb_{n,2}$ and the LKB representation found in \cite{JK}.
Equation~\ref{eq:twist_LKB_triv_ext} implies immediately that $Z(B_n) \subseteq \ker \rho_{n,\ell}^{\WH}$ whereas $Z(B_n) \not\subset \ker \rho_{n,\ell}^{\NH}$ , for $n \equiv -1 \mod r$.

\subsection{The case \texorpdfstring{$\ell=2$}{l=2} and \texorpdfstring{$n \equiv -2 \mod r$}{n=-2 mod r}} \label{sec:Nn2}

We proceed as in Section \ref{sec:N20} in order to describe the $B_n$-action on $\NH_{n,2,0}$. We solve the linear system \eqref{eq:lin_sys}. Since $\beta_j = 1$, we get that:
\[
\sum_{i=1}^{n} \lambda_i s^{-2i} = 0 \quad \text{ for all } j=1,\ldots,n.
\]
It is easy to see that $\lambda_j = s^{j-n}$, $\lambda_n = - s^{-(j-n)}$ and $\lambda_i = 0$ for $i \neq j$ and $i, j \neq n$ is a solution for the system for every $1 \leq j \leq n-1$. Therefore, by Theorem~\ref{thm:Nnl}, we have that for $n \equiv -2 \mod r$:
\[
\NH_{n,2,1} = \spn \{ b'_j := s^{j-n} \, b_j - s^{-(j-n)} \, b_n \mid 1 \leq j \leq n-1\} \oplus \WH_{n,2} , \; \text{(as vector spaces)}.
\]
Using now \eqref{eq:action_BH_4} we calculate the action of $B_n$ on the basis vectors $b'_j$:
\begin{align*}
\sigma_i.b'_j &= b'_j & \text{ for } j \neq i, i+1,
\\
\sigma_i.b'_i &= s^{i-n-3} (1-q^2) \, w_{i,i+1} + (1-s^{-2}) \, b'_i + s^{-1} \, b'_{i+1} & \text{ for } i \neq n-1,
\\
\sigma_i.b'_{i+1} &= s^{-1} \, b'_i & \text{ for } i \neq n-1,
\\
\sigma_{n-1}.b'_{j} &= b'_{j} - s^{n-j-1} \, b'_{n-1} & \text{ for } j \neq n-1,
\\
\sigma_{n-1}.b'_{n-1} &= s^{-4} (1-q^2) \, w_{n-1,n} - s^{-2} \, b'_{n-1}.
\end{align*}
The choice of the solution of the linear system \eqref{eq:lin_sys} is canonical with respect to the representation $\WH_{n,1}$, in the sense that $\{ E \, b_j' \mid j=1,\ldots,n-1\}$ is the basis of $\WH_{n,1}$ as in Lemma~\ref{lem:Wnl_Anl}:
\[
E \, b_j' = c_j - s^{n-i} c_n = \Phi(a_i) \andq E \, w = 0 \; \forall w \in \WH_{n,2}.
\]
Moreover, by Theorem~\ref{thm:faithful_center}:
\[
\theta_n b'_j = q^{-4} b'_j - s^n q^{-6} FE \, b'_j
\andq
\theta_n w_{i,j} = q^{-4}w_{i,j}.
\]
We see immediately that $\langle \theta_n^m \rangle \subseteq \ker \rho_{n,\ell}^{\WH}$, for some $m \in \Np$ (depending on $r$), whereas $Z(B_n) \not\subset \ker \rho_{n,\ell}^{\NH}$.

\subsection{The case \texorpdfstring{$\ell=2$}{l=2} and non-faithfulness for \texorpdfstring{$n \geq 3$}{n>=3} and for \texorpdfstring{$r \geq 5$}{r>=5}}
In this section we prove that the representations $\NH_{n,2,0}$ and $\NH_{n,2,1}$ are not faithful when $n \geq 3$ and $r \geq 5$. 

The LKB representation is known to be equivalent to an irreducible representation of the Birman--Murakami--Wenzl (BMW) algebra \cite{Zi}. The generators of the BMW algebra corresponding to the generators of $B_n$ satisfy a cubic relation (see also \cite{MW2}). Therefore, the matrices of the representation $\Wb_{n,2}$ for $n \geq 3$ corresponding to $\sigma_i \in B_n$ also satisfy a cubic relation (for $n=2$ we have that $\dim \Wb_{n,2} = 1$). In fact, we have that
\[
(\sigma_i - 1) (\sigma_i + s^{-2}) (\sigma_i - s^{-4} q^2) = 0, \forq 1 \leq i < n.
\]
The above equation can be easily proven by computing the eigenvalues of the matrix for $\sigma_1$ for $n=3$. Hence
\[
p(X) := (X - 1) (X + s^{-2}) (X - s^{-4} q^2)
\]
is the minimal polynomial satisfied by the matrices of $\Wb_{n,2}$ for $n \geq 3$. 

Note that $1$ is the single eigenvalue of the trivial representation and $1, -s^{-2}$ are the eigenvalues of the Burau representation $\Wb_{n,1}$. Moreover, the matrices of $\NH_{n,2,0}$ and $\NH_{n,2,1}$ (in an ordered basis consisting of the head vectors in $\HH_{n,2}$ followed by the vectors of $\WH_{n,2}$) are lower triangular, with two diagonal blocks (one for the trivial or the Burau representation and one for the LKB representation). Therefore, the eigenvalues of $\NH_{n,2,0}$ and $\NH_{n,2,1}$ are exactly $1, -s^{-2}= q^{r+2}$ and $s^{-4} q^2 = q^6$. Moreover, if the three eigenvalues, which are now complex numbers, are distinct, then the minimal polynomial of the matrices of $\WH_{n,2}$ is again the polynomial $p(X)$. For the eigenvalues to be distinct, it has to hold that
\begin{itemize}
\item $q^{r+2} \neq 1 \Leftrightarrow r \geq 3$;
\item and $q^6 \neq 1 \Leftrightarrow r \geq 4$;
\item and $q^{r+2} \neq q^6 \Leftrightarrow q^{r-4} \neq 1 \Leftrightarrow r \geq 5$.
\end{itemize}
Now, we can conclude the following, by proving that the matrices of $\NH_{n,2,0}$ and $\NH_{n,2,1}$ also satisfy the polynomial $p(X)$:

\begin{lemma} \label{lem:min_pol}
Let $n \in \N_{\geq 3}$, $r \in \N_{\geq 5}$.
The minimal polynomial of the matrices of $\NH_{n,2,0}$ and $\NH_{n,2,1}$ is the polynomial $p(X)$ (with $q=e^{\pi \im / r}$ and $s=q^{r-1}$).
\end{lemma}

The proof of Lemma~\ref{lem:min_pol} can be found in Appendix~\ref{app:min_pol}. We now proceed to the main statement of this section.

\begin{prop}
Let $n \in \N_{\geq 3}$, $r \in \N_{\geq 5}$.
The representations $\NH_{n,2,0}$ and $\NH_{n,2,1}$ are not faithful. In particular, we have that for every $1 \leq i < n$:
\[
\begin{array}{rl}
\langle \sigma_i^r \rangle \subset \ker \rho^{\NH}_{n,2}, & \text{ if } r \text{ is even,}\\
\langle \sigma_i^{2r} \rangle \subset \ker \rho^{\NH}_{n,2}, & \text{ if } r \text{ is odd.}
\end{array}
\]
\end{prop}

\begin{proof}
We prove the statement for $\sigma_1 \in B_n$; then it follows immediately for the rest of the generators of $B_n$, since they are all conjugate to $\sigma_1$. By Lemma~\ref{lem:min_pol} the minimal polynomial for the matrix of $\sigma_1$ in the representation $\NH_{n,2,0}$ or $\NH_{n,2,1}$ is the cubic polynomial $p(X)$, for $r \geq 5$. By the discussion above the lemma, the eigenvalues are distinct complex numbers for $r \geq 5$ and the minimal polynomial is factorized into linear factors. Therefore, there exists a basis, in which the matrix for $\sigma_1$ is a diagonal matrix with the eigenvalues as diagonal entries (which may be repeated according to their multiplicities).

Since the eigenvalues ($1$, $q^6$ and $q^{r+2}$) are $2r$-roots of unity, it is immediate that $\langle \sigma_1^{2r} \rangle \subset \ker \rho^{\NH}_{n,2}$. In the case that $r$ is even, we write $r=2r'$, for some $r' \in \Np$. We have obviously that $q^{6r} = 1$. Further, $q^{r(r+2)} = q^{2r'(2r'+2)} = q^{4r'(r'+1)} = 1$, which shows that $\langle \sigma_1^r \rangle \subset \ker \rho^{\NH}_{n,2}$, if $r$ is even.
\end{proof}

\begin{rem}
Note that the representations $\NH_{2,2,0}$ and $\NH_{2,2,1}$ are faithful by Theorem~\ref{thm:faithful_center}, since $B_2 = Z(B_2)$.
\end{rem}

\section{Generalization to 3-variable representations} \label{sec:3var}

After having explicitly described the $B_n$-action for $\ell=2$, a natural question is whether these representations can be generalized to $2$-variable representations extending the LKB representation. In this section, we prove that they can actually be generalized to $3$-variable representations. We focus mainly on the representation extending the trivial representation by the LKB representation, for which we prove that it splits as a direct sum if we extend the defining ring. Further, we study a certain specialization of this representation at roots of unity. Finally, we study the restriction of $\NH_{n,2,0}$ on $B_{n-1}$.

Let
\[
\Lb := \Z[q^{\pm 1}, s^{\pm 1}] \andq \Lf := \Q(q,s).
\]
Note that $\Lf$ is the field of fractions of $\Lb$. Recall also that the representations $\Wb_{n,\ell}$ are irreducible over $\Lf$ \cite{JK}.

\subsection{The 3-variable generalizations}
In this section we define the representations $\widetilde{\NH}_{n,2,0}$ and $\widetilde{\NH}_{n,2,1}$ generalizing $\NH_{n,2,0}$ and $\NH_{n,2,1}$ respectively. The representation $\widetilde{\NH}_{n,2,0}$ extends $\WH_{n,0}$ by $\Wb_{n,2}$ (note that $\Wb_{n,0} \cong \WH_{n,0}$ since both are the trivial representation) and the representation $\widetilde{\NH}_{n,2,0}$ extends $\Wb_{n,1}$ by $\Wb_{n,2}$. In both cases we let $q$ and $s$ to be variables and introduce a third variable $t$.

\begin{prop} \label{prop:Nn20}
Let $\widetilde{\NH}_{n,2,0}$ be a free $\Lb[t]$-module spanned by $\binom{n}{2}+1$ vectors denoted by $b, w_{1,2}, \ldots, w_{1,n}, w_{2,3}, \ldots, w_{2,n}, \ldots, w_{n-1,n}$. The space $\widetilde{\NH}_{n,2,0}$ is a $B_n$-representation where the $B_n$-action on the vectors $w_{i,j}$ is defined by \eqref{eq:LKB} and on the vector $b$ by:
\begin{equation} \label{eq:sigma_i_b}
\sigma_i b = b + t \, w_{i,i+1} .
\end{equation}
\end{prop}

To prove the result, one needs to check that the braid group relations hold for the vector $b$. The complete proof can be found in Appendix~\ref{app:Nn20}. Note that the representation $\widetilde{\NH}_{n,2,0}$ is faithful since its subrepresentation $\Wb_{n,2}$ is faithful \cite{Bi1,Kr2}. 

Analogously we define as follows the representation $\widetilde{\NH}_{n,2,1}$.

\begin{prop} \label{prop:Nn21}
Let $n \in \Npp$ and let $\widetilde{\NH}_{n,2,1}$ be a free $\Lb[t^{\pm 1}]$-module spanned by $\binom{n}{2}+n-1$ vectors $b'_1, \ldots, b'_{n-1}, w_{1,2}, \ldots, w_{1,n}, w_{2,3}, \ldots, w_{2,n}, \ldots, w_{n-1,n}$. The space $\widetilde{\NH}_{n,2,1}$ is a $B_n$-representation where the $B_n$-action on the vectors $w_{i,j}$ is defined by \eqref{eq:LKB} and on the vectors $b'_j$ by:
\begin{align*}
\sigma_i b'_j &= b'_j & \text{ for } j \neq i, i+1,
\\
\sigma_i b'_i &= s^{i-n} t \, w_{i,i+1} + (1-s^{-2}) \, b'_i + s^{-1} \, b'_{i+1} & \text{ for } i \neq n-1,
\\
\sigma_i b'_{i+1} &= s^{-1} \, b'_i & \text{ for } i \neq n-1,
\\
\sigma_{n-1} b'_{j} &= b'_{j} - s^{n-j-1} \, b'_{n-1} & \text{ for } j \neq n-1,
\\
\sigma_{n-1} b'_{n-1} &= s^{-1} t \, w_{n-1,n} - s^{-2} \, b'_{n-1}.
\end{align*}
\end{prop}
To prove the theorem one needs to check that the braid group relations are satisfied for the vectors $b'_m$, with $1 \leq m \leq n-1$. The complete proof can be found in Appendix~\ref{app:Nn21}. 

\subsection{Specializations of \texorpdfstring{$\widetilde{\NH}_{n,2,0}$}{NH(n,2,0)} at roots of unity} \label{sec:gen_N20}

Let $r \in \N_{\geq 3}$, $q=e^{\pi \im / r}$, $s = q^{r-1}$ and $t = s^{-3} (1-q^2)$ and let $\widetilde{\NH}_{n,2,0}(q,s,t)$ be corresponding specialization of the representation $\widetilde{\NH}_{n,2,0}$. By Prop.~\ref{prop:Nn20} the representation $\NH_{n,2,0}$ extends $\WH_{n,0}$ non trivially by $\WH_{n,2}$ if and only if $n \equiv -1 \mod r$. On the other hand the $B_n$-representation $\widetilde{\NH}_{n,2,0}(q,s,t)$ is defined for any $n \in \Npp$ and is isomorphic to $\NH_{n,2,0}$ for $n \equiv -1 \mod r$. Therefore, the representation $\widetilde{\NH}_{n,2,0}(q,s,t)$ could provide us a non trivial extension of $\WH_{n,0}$ by $\WH_{n,2}$ for any $n \in \Npp$. The goal of this section is to prove the following.

\begin{thm} \label{thm:Nn2_split}
Let $n \in \Npp$, $r \in \N_{\geq 3}$ and let $q=e^{\pi \im / r}$, $s = q^{r-1}$ and $t = s^{-3} (1-q^2)$. The short exact sequence of $B_n$-modules
\[
\begin{array}{ccccccccc}
0 & \longrightarrow & \WH_{n,2} & \lhook\joinrel\longrightarrow & \widetilde{\NH}_{n,2,0}(q,s,t) &\longrightarrow  & \widetilde{\NH}_{n,2,0}(q,s,t) / \WH_{n,2} & \longrightarrow & 0
\end{array}
\]
does not split if and only if $n \equiv -1 \mod r$.
\end{thm}

To do that we start by investigating under which conditions the representation $\widetilde{\NH}_{n,2,0}(q, s, t)$ splits as a direct sum of the specialized LKB and the trivial representations for any $q, s, t \in \C^\times$.

\begin{prop} \label{prop:Nn20_split_spec}
Let $n \in \Npp$ and $q, s, t  \in \C^\times$. The short exact sequence of $B_n$-modules
\[
\begin{array}{ccccccccc}
0 & \longrightarrow & \Wb_{n,2}(q,s) & \lhook\joinrel\longrightarrow & \widetilde{\NH}_{n,2,0}(q,s,t) &\longrightarrow  &  \widetilde{\NH}_{n,2,0}(q,s,t) / \Wb_{n,2}(q,s) & \longrightarrow & 0
\end{array}
\]
splits if and only if
\[
s^2 = 1 \; \text{ and } \; q^2 \neq 1,  \qquad \text{ or } \qquad
s^2 \neq 1 \; \text{ and } \; q^2 \neq s^{2n}.
\]
\end{prop}

\begin{proof}
We realise the trivial representation $\WH_{n,0} \simeq \widetilde{\NH}_{n,2,0}(q,s,t) / \Wb_{n,2}(q,s)$ as a direct summand of $\widetilde{\NH}_{n,2,0}(q,s,t)$. For this, we find a change of basis of $\widetilde{\NH}_{n,2,0}(q,s,t)$ such that the $B_n$-action on $\WH_{n,0}$ is closed. That is, we find a $w = \sum_{1 \leq i < j \leq n} \lambda_{i,j} \, w_{i,j} \in \Wb_{n,2}(q,s)$, where $\lambda_{i,j} \in \C$, such that \mbox{$(\sigma_k -1)(b + w) = 0$}, for $1 \leq k \leq n-1$.

Due to \eqref{eq:sigma_i_b} the vector $w$ should satisfy the equations
\begin{equation}\label{eq:sigma_k_w}
(\sigma_k -1)w = - (\sigma_k -1) b = -t \, w_{k,k+1} \forq 1 \leq k \leq n-1.
\end{equation}
Let the set
\[
I_k := \{ (i,j) \mid 1 \leq i < j \leq n -1 \text{ and } \{ i, j \} \cap \{ k , k+1 \} = \emptyset  \}.
\]
For $1 \leq k \leq n-1$, we have that:
\begin{align*}
&\sigma_k w = \sum_{1 \leq i < j \leq n} \lambda_{i,j} \, \sigma_k w_{i,j}
= \sum_{(i,j) \in I_k} \lambda_{i,j} \, \sigma_k w_{i,j}
+ \sum_{j=k+2}^n \lambda_{k+1,j} \, \sigma_k w_{k+1,j}
\\
&+ \sum_{i=1}^{k-1} \lambda_{i,k+1} \, \sigma_k w_{i,k+1}
+ \sum_{j=k+2}^n \lambda_{k,j} \, \sigma_k w_{k,j}
+ \sum_{i=1}^{k-1} \lambda_{i,k} \, \sigma_k w_{i,k}
+ \lambda_{k,k+1} \, \sigma_k w_{k,k+1}.
\end{align*}
Note that the summands in the above expression correspond exactly to the cases (depending on $i,j$ and $k$) for the action of $\sigma_k$ on a vector $w_{i,j}$ \eqref{eq:LKB}. Therefore, by \eqref{eq:LKB} we get:
\begin{align*}
&\sigma_k w = w
+  s^{-1} \sum_{j=k+2}^n \left( \lambda_{k+1,j} - s^{-1} \lambda_{k,j} \right) w_{k,j}
+ s^{-1} \sum_{i=1}^{k-1} \left(  \lambda_{i,k+1} - s^{-1} \lambda_{i,k} \right) w_{i,k}
\\
&- \sum_{j=k+2}^n \left( \lambda_{k+1,j} - s^{-1} \lambda_{k,j} \right) w_{k+1,j}
- \sum_{i=1}^{k-1} \left( \lambda_{i,k+1} - s^{-1} \lambda_{i,k} \right) w_{i,k+1}
+ A_k w_{k,k+1},
\end{align*}
where, for $1 \leq k \leq n-1$,
\[
A_k := (s^{-4}q^2-1) \lambda_{k,k+1} - (1 - s^{-2}) \left( \sum_{i=1}^{k-1} s^{k-i-1} \lambda_{i,k} 
+ q^2 \sum_{j=k+2}^n s^{k-j-1}  \lambda_{k,j} \right).
\]
Due to \eqref{eq:sigma_k_w} we have for every $1 \leq k \leq n-1$:
\begin{align}
\lambda_{k+1,j} &= s^{-1} \lambda_{k,j} \forq k+2 \leq j \leq n, \label{eq:system_0_1}
\\
\lambda_{i,k+1} &= s^{-1} \lambda_{i,k} \forq 1 \leq i \leq k-1, \label{eq:system_0_2}
\\
A_k &= -t. \label{eq:system_0_3}
\end{align}
By \eqref{eq:system_0_1} and \eqref{eq:system_0_2} we have
\begin{equation} \label{eq:system_0_rec_1}
\lambda_{i,j} = s^{2n-i-j-1} \lambda_{n-1,n} \forq 1 \leq i < j \leq n.
\end{equation}
Therefore, it remains to investigate under which conditions Equations~\ref{eq:system_0_1}, \ref{eq:system_0_2} and \ref{eq:system_0_3} imply $\lambda_{n-1,n} \neq 0$. By \eqref{eq:system_0_3} for $k=n-1$ and \eqref{eq:system_0_rec_1} we have
\begin{equation} \label{eq:system_0_An1}
A_{n-1} = (s^{-4}q^2-1) \lambda_{n-1, n} - (1 - s^{-2}) \sum_{i=1}^{n-2} s^{2(n-i-1)} \lambda_{n-1,n} = -t.
\end{equation}
If $s^2 = 1$, then \eqref{eq:system_0_An1} becomes
\[
(q^2-1) \lambda_{n-1, n} = -t.
\]
The equation is satisfied if and only if $q^2 \neq 1$ and then we obtain
\[
\lambda_{n-1, n} = \frac{t}{1-q^2}.
\]
On the other hand, if $s^2 \neq 1$ then by \eqref{eq:system_0_An1} we have
\[
\begin{aligned}
A_{n-1} &= (s^{-4}q^2-1) \lambda_{n-1, n} - (1-s^{-2}) \frac{s^2 - s^{2(n-1)}}{1-s^2} \lambda_{n-1, n}
\\
&= (s^{-4}q^2 - s^{2n-4}) \lambda_{n-1,n} = -t.
\end{aligned}
\]
Since $s^2 \neq 1$, the equation is satisfied if and only if $q^2 \neq s^{2n}$ and then we obtain
\[
\lambda_{n-1,n} = \frac{s^4 t}{s^{2n}-q^2}.
\]

Therefore, the short exact sequence splits if and only if $s^2 = 1$ and $q^2 \neq 1$, or if $s^2 \neq 1$ and $q^2 \neq s^{2n}$.
\end{proof}

We are now ready to prove the main result of this section.
\begin{proof}[Proof of  Thm.~\ref{thm:Nn2_split}]
Recall that $r \in \N_{\geq 3}$, $q=e^{\pi \im / r}$, $s = q^{r-1}$ and $t = s^{-3} (1-q^2)$. We apply Prop.~\ref{prop:Nn20_split_spec} on $\widetilde{\NH}_{n,2,0}(q,s,t)$. Since $r \geq 3$ we have that $s^2 \neq 1$. Further:
\[
q^2 = s^{2n}
\; \Leftrightarrow \;
q^{2n+2} = 1
\; \Leftrightarrow \;
2n + 2 \equiv 0 \mod 2r
\; \Leftrightarrow \;
n \equiv -1 \mod r.
\]
Therefore, the short exact sequence of the statement does not split if and only if $n \equiv -1 \mod r$.
\end{proof}

Note that, Theorem~\ref{thm:Nn2_split} provides an alternative proof for the fact that $\NH_{n,2,0}$ does not split (cf. Theorem~\ref{thm:thm1}).

\begin{rem}
The representation $\NH_{n,2,0}$ is defined by Theorem~\ref{thm:Nnl} for $r \geq 3$, since $\ell=2$. Moreover, the short exact sequence of Theorem~\ref{thm:Nn2_split} is split for every $n \in \Npp$ when $r=2$ by Prop.~\ref{prop:Nn20_split_spec}, since $s^2=q^2=1$. Therefore, there is no analogous non trivial extension of $\WH_{n,0}$ by $\WH_{n,2}$ for $r=2$.
\end{rem}

\subsection{The representation \texorpdfstring{$\widetilde{\NH}_{n,2,0}$}{NH(n,2,0)} as a direct sum}
In this section we prove that the representation $\widetilde{\NH}_{n,2,0}$ splits a direct sum of the LKB and the trivial representations by adapting the proof of Prop.~\ref{prop:Nn20_split_spec}.

\begin{prop} \label{prop:Nn20_split}
Let $n \in \Npp$. The short exact sequence of $B_n$-modules
\[
\begin{array}{ccccccccc}
0 & \longrightarrow & \Wb_{n,2} & \lhook\joinrel\longrightarrow & \widetilde{\NH}_{n,2,0} &\longrightarrow  & \widetilde{\NH}_{n,2,0} / \Wb_{n,2} & \longrightarrow & 0,
\end{array}
\]
splits over $\Lf[t]$.
\end{prop}

\begin{proof}
To show that the sequence splits, we need to find a vector $0 \neq w = \sum_{1 \leq i < j \leq n} \lambda_{i,j} \, w_{i,j} \in \Wb_{n,2}$, where $\lambda_{i,j} \in \Lf[t]$, such that \mbox{$(\sigma_k -1)(b + w) = 0$}, for $1 \leq k \leq n-1$. Proceeding as in Proof of Prop.~\ref{prop:Nn20_split_spec} by considering $q$, $s$ and $t$ as variables, we get
\[
\lambda_{i,j} = s^{2n-i-j-1} \lambda_{n-1,n} \quad \text{ where } \quad \lambda_{n-1,n} = \frac{s^4 t}{s^{2n}-q^2}.
\]
The computations are exactly as in Proof of Prop.~\ref{prop:Nn20_split_spec}, with the difference that we do not need to check whether we divide by zero, since $s$, $q$ and $t$ are now variables. Since $\lambda_{n-1,n} \in \Lf[t]$ but $\lambda_{n-1,n} \not\in \Lb[t]$, the representation $\widetilde{\NH}_{n,2,1}$ splits over $\Lf[t]$ and not over $\Lb[t]$.
\end{proof}

\subsection{Restricting \texorpdfstring{$\NH_{n,2,0}$}{NH(n,2)} on \texorpdfstring{$B_{n-1}$}{B(n-1)}}
Another means to construct a non trivial extension of $\WH_{n,0}$ by $\WH_{n,2}$ for any $n \in \Npp$ is by restricting $\NH_{n,2,0}$ on $B_{n-1}$. 
In this section we prove using Theorem~\ref{thm:Nn2_split} that this restriction splits.

Before we start, we recall the following facts considering the restriction of $\Wb_{n,\ell}$ on $B_{n-1}$ \cite{JK} that we apply directly on $\WH_{n, \ell}$. The map $\Vs_{r-1}^{\otimes n-1} \to \Vs_{r-1}^{\otimes n}$ defined by $u \mapsto u_0 \otimes u$ induces an embedding $\iota_{n-1,\ell}^{\WH} \colon \WH_{n-1,\ell} \hookrightarrow \WH_{n,\ell}$. The basis of $\WH_{n-1,\ell}$ with respect to the basis of $\WH_{n,\ell}$ consists of the vectors $\Phi(a_{\vec{\varepsilon}})$ \eqref{eq:map_phi}, where $a_{\vec{\varepsilon}} \in \AH_{n,\ell}$ as in \eqref{eq:a_epsilon} with $\vec{\varepsilon} = (\varepsilon_j, \ldots, \varepsilon_n)$ with $j > 2$. Let now the inclusion
\begin{align*}
q_{n-1} \colon B_{n-1} &\hookrightarrow B_n
\\
\sigma_i &\mapsto \sigma_{i+1}.
\end{align*}
The inclusion $\iota_{n-1,\ell}^{\WH}$ is $B_{n-1}$-equivariant with respect to $q_{n-1}$. In other words, the action of $q_{n-1}(B_{n-1}) = \langle \sigma_2, \ldots, \sigma_{n-1} \rangle \subset B_n$ on $\WH_{n,\ell}$ is reducible.

For $\ell=2$ the basis of $\WH_{n-1,2}$ with respect to the basis of $\WH_{n,2}$ is given by the vectors $w_{i,j}$, where $2 \leq i < j \leq n$. Further, note that for $\ell = 0$, it holds $\WH_{n,0} = \spn \{ u_0^{\otimes n} \}$ for every $n \in \Npp$. We denote by $\Res_{n-1} W$, where $W$ a $B_n$-module, the restriction of $W$ on $B_{n-1}$ with respect to the inclusion $q_{n-1} \colon B_{n-1} \hookrightarrow B_n$. Now, we have the following:

\begin{prop}
Let $n \in \Npp$, $r \in \N_{\geq 3}$ and suppose that $n \equiv -1 \mod r$. The map
\begin{align*}
\iota_{n-1,2}^{\NH} \colon \NH_{n-1,2,0} &\lhook\joinrel\longrightarrow \NH_{n,2,0}
\\
w_{i,j} &\longmapsto \iota_{n-1,2}^{\WH}(w_{i,j}) = w_{i+1,j}
\\
b &\longmapsto b,
\end{align*}
is a $B_{n-1}$-equivariant embedding with respect to the inclusion $q_{n-1} \colon B_{n-1} \hookrightarrow B_n$.
Moreover, the short exact sequence of $B_{n-1}$-modules
\[
\begin{array}{ccccccccc}
0 & \longrightarrow & \WH_{n,2} & \lhook\joinrel\longrightarrow & \Res_{n-1} \NH_{n,2,0} &\longrightarrow  & \WH_{n,0} & \longrightarrow & 0
\\
& & w_{i,j} & \longmapsto & w_{i,j} & \longmapsto & 0 & &
\\
& & & & b & \longmapsto & b, & &
\end{array}
\]
splits.
\end{prop}

\begin{proof}
Since $\ell = 2$ the representations $\NH_{n,2,0}$ are defined by Theorem~\ref{thm:Nnl} for $2 = \ell < r$ and therefore we have the condition $r \geq 3$ of the statement.

Note that, the map $\iota_{n-1,2}^{\NH}$ is by definition $B_{n-1}$-equivariant on $\WH_{n-1,2} \subset \NH_{n-1,2,0}$. By the action of $B_n$ on $b$ \eqref{eq:action_vec_b} we have
\[
\iota_{n-1,2}^{\NH} (\sigma_i \, b)
= \iota_{n-1,2}^{\NH} (b + w_{i,i+1})
= b + w_{i+1,i+2}
=  \sigma_{i+1} \, b
= q_{n-1} (\sigma_i) \, \iota_{n-1,2}^{\NH} (b).
\]
Therefore, the map $\iota_{n-1,\ell}^{\NH}$ is a $B_{n-1}$-equivariant map with respect to $q_{n-1}$.

By the above and the $B_{n-1}$-equivariant embeddings $\iota_{n-1,\ell}^{\WH} \colon \WH_{n-1,\ell} \hookrightarrow \WH_{n,\ell}$, we get the commuting diagram of $B_{n-1}$-modules
\[
\begin{tikzcd}
0 \arrow[r] & \WH_{n,2} \arrow[r, hook] & \Res_{n-1} \NH_{n,2,0} \arrow[r] & \WH_{n,0} \arrow[r] & 0
\\
0 \arrow[r] & \WH_{n-1,2} \arrow[u, hook, "\iota_{n-1,2}^{\WH}"] \arrow[r, hook] & \NH_{n-1,2,0} \arrow[u, hook, "\iota_{n-1,2}^{\NH}"] \arrow[r] & \WH_{n-1,0} \arrow[u, hook, "\iota_{n-1,0}^{\WH}"] \arrow[r] & 0,
\end{tikzcd}
\]
where the maps of the second short exact sequence are defined by
\[
\begin{array}{ccccccccc}
0 & \longrightarrow & \WH_{n-1,2} & \lhook\joinrel\longrightarrow & \NH_{n-1,2,0} &\longrightarrow  & \WH_{n-1,0} & \longrightarrow & 0
\\
& & w_{i,j} & \longmapsto & w_{i,j} & \longmapsto & 0 & &
\\
& & & & b & \longmapsto & b. & &
\end{array}
\]

Since $n \equiv - 1 \mod r$, we get $n-1 \not\equiv -1 \mod r$ and therefore, by Theorem~\ref{thm:Nn2_split} the second short exact sequence (for $\NH_{n-1,2,0}$) splits. In detail, there exists a $w \in \WH_{n-1,2}$ such that $\beta (b+w) = b+w$, for any braid $\beta \in B_{n-1}$ (see Proof of Prop.~\ref{prop:Nn20_split_spec}). Therefore, due to the $B_{n-1}$-equivarance of the map $\iota_{n-1,2}^{\NH}$ we have that
\[
q_{n-1}(\beta) \, \iota_{n-1,2}^{\NH}(b+w) = \iota_{n-1,2}^{\NH}(b+w),
\quad \text{ in } \Res_{n-1} \NH_{n,2,0},
\]
for any $\beta \in B_{n-1}$. In other words, we have that
\[
\beta' (b + \iota_{n-1,2}^{\WH}(w)) = b + \iota_{n-1,2}^{\WH}(w), \quad \text{ in } \Res_{n-1} \NH_{n,2,0},
\]
for any $\beta' \in \langle \sigma_2, \ldots, \sigma_{n-1} \rangle$. That is, the first short exact sequence splits.
\end{proof}

\subsection{A specialization of \texorpdfstring{$\widetilde{\NH}_{n,2,0}$}{N(n,2,0)}}
We now consider the specialization of $\widetilde{\NH}_{n,2,0}$ at $s=q=1$, which we denote by $\overline{\NH}_{n,2,0}$. Then the subrepresentation corresponding to the LKB representation, which we denote by $\overline{\WH}_{n,2}$, is now specialized at $s=q=1$ and \eqref{eq:LKB} implies that it factors through the symmetric group. We now prove the following statement regarding to whether $\overline{\NH}_{n,2,0}$ is faithful or not.

\begin{prop} \label{prop:Delta_n20}
Let $n \in \Npp$. The representation $\overline{\NH}_{n,2,0}$ is faithful on $\langle \Delta_n \rangle$. In particular, it holds that:
\begin{equation} \label{eq:Delta_sq1}
\Delta_n^k b = b + kt \sum_{1 \leq i < j \leq n} w_{i,j}
\end{equation}
\end{prop}
The proof can be found in Appendix~\ref{app:Delta_n20}.

\begin{rem}
Note that the braid $(\sigma_i \sigma_{i+1}^{-1})^3$, where $1 \leq i < n-1$, is in the kernel of $\overline{\NH}_{n,2,0}$.
\end{rem}

\section{The representation \texorpdfstring{$\WH_{4,2}$}{WH(4,2)} for \texorpdfstring{$r=3$}{r=3} and the cubic Hecke algebra} \label{sec:cubic}

By Theorem~\ref{thm:Wnl}, for $r=3$ and $n=4$ there exists a subrepresentation $\SH_{4,2} \subset \WH_{4,2}$ isomorphic to $\WH_{4,1}$. Then the quotient $\WH_{4,2} / \SH_{4,2}$ defines a 3-dimensional $B_4$-representation. In this section we show that this representation is equivalent to a representation of the cubic Hecke algebra on 4 strands.

\smallbreak
Let $x, y, z \in \C$. The \textit{cubic Hecke algebra} $H_n$ is defined as the quotient of the group algebra $\C B_n$ over the cubic relation $(\sigma_1 - x) (\sigma_1 - y) (\sigma_1 - z) = 0$. Since all braid group generators are conjugate to each other, one can also define $H_n$ as the quotient over the relations $(\sigma_i - x) (\sigma_i - y) (\sigma_i - z) = 0$, for all $1 \leq i < n$. Note that, $H_3, H_4$ and $H_5$ are isomorphic to the generalized Hecke algebras of the complex reflection groups $G_4, G_{25}$ and $G_{32}$ respectively and are finite-dimensional \cite{BM, Mar}. The Birman-Murakami-Wenzl (BMW) algebras \cite{BW,MuBMW} are naturally related to the cubic Hecke algebras and further, $H_n$ has been used to construct a quotient related to the Links-Gould invariants \cite{MW1, MW2}.

There exists an irreducible $3$-dimensional representation of $H_4$ \cite{MarThese, MW1} defined by:
\begin{equation} \label{eq:cubic_V}
\sigma_1, \sigma_3 \mapsto \begin{pmatrix}
z & 0 & 0 \\
x z + y^2 & y & 0 \\
y & 1 & x
\end{pmatrix}
\andq 
\sigma_2 \mapsto \begin{pmatrix}
x & -1 & y \\
0 & y & - x z - y^2 \\
0 & 0 & z
\end{pmatrix}.
\end{equation}
Note that the representation factors through an $H_3$-representation \cite{MW1}. We now prove the following statement.

\begin{prop}
The representation $\WH_{4,2}/\SH_{4,2}$ for $r=3$ is equivalent to the specialization of the 3-dimensional representation of $H_4$ defined by \eqref{eq:cubic_V} with $x = q^5$ and $y=z=1$.
\end{prop}

\begin{proof}
Solving the modular condition $j \equiv n + 2(\ell-1) \mod r$, with $j < \ell$, for the representation $\WH_{4,2}$ for $r=3$, we get that $j=0$. As in the proof of Proposition~\ref{prop:struct3}, it holds that:
\[
\SH_{n,\ell} =  \Ima F^{j+1}\vert_{\WH_{n,\ell'}}, \text{ where } \ell' = \ell - 1 -j.
\]
In our case, this becomes:
\[
\SH_{4,2} =  \Ima F \vert_{\WH_{4,1}}, \text{ for } r=3.
\]

Now, for any $n \geq 2$, the vectors $\overline{c}_i := c_i - s^{n-i} c_n$, for $1 \leq j < n$, span the space $\WH_{n,1}$ \cite{JK}. Further, for $\ell=2$ and $n \equiv -2 \mod r$ we have that (see Appendix~\ref{app:phi_Burau} for the proof):
\begin{equation} \label{eq:phi_Burau}
F \, \overline{c}_i = s^{-(i-1)} q^2 \sum_{j=1}^{n-i} s^{-j} \, w_{i,i+j} + \sum_{j=1}^{i-1} s^{-(j-1)} \, w_{j,i} - s^{n-i} \sum_{j=1}^{n-1} s^{-(j-1)} \, w_{j,n}.
\end{equation}
The $B_n$-action on these vectors is given by:
\[
\begin{array}{rll}
\sigma_i F \, \overline{c}_j &= F \, \overline{c}_j & \text{ if } i \neq n-1 \text{ and } j \neq i, i+1,
\\
\sigma_i F \, \overline{c}_i &= (1 - s^{-2}) F \, \overline{c}_i + s^{-1} F \, \overline{c}_{i+1} & \text{ if } i \neq n-1,
\\
\sigma_i F \, \overline{c}_{i+1} &= s^{-1} F \, \overline{c}_i & \text{ if } i \neq n-1,
\\
\sigma_{n-1} F \, \overline{c}_{j} &= F \, \overline{c}_{j} + s^{n-j-1} \, F \, \overline{c}_{n-1} & \text{ if }j \neq n-1,
\\
\sigma_{n-1} F \, \overline{c}_{n-1} &= -s^{-2} \, F \, \overline{c}_{n-1}. &
\end{array}
\]

For $\WH_{4,2}$ and $r=3$ the following vectors span the complement of $\SH_{4,2}$ inside $\WH_{4,2}$:
\begin{align*}
g_1 &= w_{1,2}  - q^2 \, w_{1,3} + q \, w_{2,4} + w_{3,4},
\\
g_2 &= -\frac{1}{4} q^2 \left(w_{1,2}+ 2 \, w_{1,3}+w_{1,4}+w_{2,3}+2 \, w_{2,4}+w_{3,4}\right),
\\
g_3 &= -q \, w_{1,3} - w_{1,4} - w_{2,3} + \left( q - 1 \right) w_{2,4}.
\end{align*}
It remains to rewrite the action of $B_n$ on $\WH_{4,2}$ with respect to the new basis $\{g_1, g_2, g_3, F \, \overline{c}_1, F \, \overline{c}_2, F \, \overline{c}_3\}$. We write the $B_n$-action on the vectors $g_1, g_2$ and $g_3$ as:
\[
\sigma_i g_j = \sum_{k=1}^3 a_{i,j,k} \, g_k + \sum_{k=1}^3 \, b_{i,j,k} \, F \, \overline{c}_k,
\]
where $a_{i,j,k}, b_{i,j,k} \in \C$.
Denote by $[g_j]$ the images of $g_j$, for $j=1,2,3$, in the quotient $\WH_{4,2} / \SH_{4,2}$. Then the $B_n$-action on the quotient is given by:
\[
\sigma_i [g_j] = \sum_{k=1}^3 a_{i,j,k} \, [g_k].
\]
Using the Mathematica program \texttt{N42\_cubic\_Hecke.nb} (available at \cite{prog}) and \eqref{eq:LKB} we compute the coefficients $a_{i,j,k}$ and $b_{i,j,k}$ and we get that the matrices (on the basis $\{ [g_1], [g_2], [g_3] \}$) corresponding to the generators of $B_4$ acting on the quotient $\WH_{4,2} / \SH_{4,2}$ are:
\[
\sigma_1, \sigma_3 \mapsto
\begin{pmatrix}
 1 & 0 & 0 \\
 1+q^5 & 1 & 0 \\
 1 & 1 & q^5 \\
\end{pmatrix}
\andq 
\sigma_2 \mapsto 
\begin{pmatrix}
 q^5 & -1 & 1 \\
 0 & 1 & -1-q^5 \\
 0 & 0 & 1 \\
\end{pmatrix}.
\]
These are exactly the matrices of \eqref{eq:cubic_V} by substituting $x=q^5$ and $y=z=1$.
\end{proof}

\newpage
\appendix

\section{Proofs of various statements}

\subsection{Action of \texorpdfstring{$R$}{R}-matrix on weight modules} \label{app:Rmatrix_weight}
We show that the action of the $R$-matrix \eqref{eq:Rmatrix} on any (weight) module $V$ of $D$ is given by \eqref{eq:Rmatrix_unrolled}:
\[
R = q^{H \otimes H/2} \sum_{n=0}^{r-1} \frac{\qnum{1}^{2n}}{\qnum{n}!} q^{n(n-1)/2}  \left( E^n  \otimes F^n \right).
\]

Due to \eqref{eq:dq_rels} we have that $E^i k = q^{-i} \, k E^i$ and $F^i k = q^{i} \, k F^i$, for any $i \in \N$. Therefore \eqref{eq:Rmatrix} can be written as $R = C \circ \widetilde{R}$, where:
\[
C = \frac{1}{4r}  \sum_{m,m'=0}^{4r-1}  q^{- m m' / 2} k^m \otimes k^{m'} 
\andq
\widetilde{R} = 
\sum_{n=0}^{r-1} \frac{\qnum{1}^{2n}}{\qnum{n}!} q^{n(n-1)/2}  \left( E^n  \otimes F^n \right).
\]
Therefore, it remains to prove that $C = q^{H \otimes H/2}$ \eqref{eq:qHH2} as operators acting on $V \otimes V$.
Let $-2r \leq \lambda, \lambda' < 2r$ and let $v$ and $w$ be weight vectors of $V$, such that $k v = q^{\lambda/2} v$ and $k w = q^{\lambda'/2} w$. We have that:
\begin{align*}
4r \, C(v \otimes w) = \sum_{m,m'=0}^{4r-1}  q^{- m m' / 2} k^m \, v \otimes k^{m'} \, w
&= \sum_{m,m'=0}^{4r-1}  q^{(- m m' + m \lambda + m' \lambda') /2}  \, v \otimes w
\\
&= \sum_{m=0}^{4r-1} q^{m \lambda /2} \sum_{m'=0}^{4r-1} q^{\frac{\lambda'-m}{2}m'}  \, v \otimes w.
\end{align*}
Note that $q^{\frac{\lambda'-m}{2}}=1$ if and only if $\lambda'-m = 0 \mod 4r$, since $q$ is a primitive $2r$-th root of unity. Since $0 \leq m \leq 4r-1$ this happens only for one such $m$. Denote this value of $m$ by $m_0 = 4 \mu r + \lambda'$, for some $\mu \in \Z$. Then at the second summation (over $m'$) for all $m' \neq m_0$ the corresponding summand equals 0, since it is multiple of the sum of all $4r/\gcd(4r,\lambda'-m)$-roots of unity. Therefore:
\begin{align*}
\sum_{m,m'=0}^{4r-1}  q^{- m m' / 2} k^m \, v \otimes k^{m'} \, w
&= 4r \; q^{ m_0 \lambda/2} q^{(\lambda' - m_0) /2} v \otimes w
\\
&= 4r \, q^{\lambda \lambda' /2} v \otimes w
= 4r \, q^{H \otimes H/2} (v \otimes w),
\end{align*}
that is, $C = q^{H \otimes H/2}$ on $V \otimes V$.

\subsection{\texorpdfstring{$R$}{R}-matrix on \texorpdfstring{$\Vs_{r-1}^{\otimes 2}$}{V(r-1) tensor V(r-1)}} \label{app:Rmatrix_V}
Applying the $R$-matrix \eqref{eq:Rmatrix_unrolled} composed with the permutation operator $\tau$ on a vector $u_i \otimes u_j$ of $\Vs_{r-1}^{\otimes 2}$ we get:
\begin{align*}
(\tau \circ R)(u_i \otimes u_j) = q^{H \otimes H /2}\sum_{n=0}^{r-1} \frac{\qnum{1}^{2n}}{\qnum{n}!} q^{n(n-1)/2} F^n \, u_j \otimes E^n \, u_i
\end{align*}
Note that $E^n \,u_i \neq 0$ if $n < i +1$ and $F^n \, u_j \neq 0$ if $n<r-j$. Hence, the summation is up to $\min (i,r-j-1)$ and all other terms are zero. Substituting the action of $E$ and $F$ we get:
\begin{equation}\label{eq:Rcalc_1}
\begin{aligned}
(\tau \circ R)(u_i \otimes &u_j) = \sum_{n=0}^{\min (i,r-j-1)} \Bigg( \frac{\qnum{1}^{2n}}{\qnum{n}!} q^{\frac{n(n-1)}{2}}
\\
& \cdot \prod_{m=0}^{n-1} \qint{m+j+1} \qint{r - 1 -(m+j)}\, q^{H \otimes H/2} (u_{j+n} \otimes u_{i-n}) \Bigg).
\end{aligned}
\end{equation}
It holds that:
\begin{equation}\label{eq:Rcalc_2}
\frac{\qnum{1}^n}{\qnum{n}!} \prod_{m=0}^{n-1} \qint{m+j+1} = (\qint{n}!)^{-1} \prod_{m=1}^{n} \qint{m+j} = \frac{\qint{n+j}!}{\qint{j}!\, \qint{n}!} = \qbin{n+j}{j}.
\end{equation}
Using the equality $\qnum{1} \cdot \qint{r - 1 -(m+j)} = -\qnum{-1-(m+j)} = \qnum{m+j+1}$ and substituting Equation~\eqref{eq:Rcalc_2} into Equation~\eqref{eq:Rcalc_1} we get:
\begin{align}\label{eq:Rcalc_3}
\begin{split}
(\tau \circ R)&(u_i \otimes u_j) = \\
&\sum_{n=0}^{\min (i,r-j-1)} q^{n(n-1)/2} \qbin{n+j}{j} \prod_{m=0}^{n-1} \qnum{m+j+1} \, q^{H \otimes H /2} (u_{j+n} \otimes u_{i-n}).
\end{split}
\end{align}
We set $s:=q^{r-1}=-q^{-1}$ \eqref{eq:s}.  The action of $q^{H \otimes H /2}$ on $u_{j+n} \otimes u_{i-n}$ is as follows:
\begin{align*}
q^{H \otimes H /2} (u_{j+n} \otimes u_{i-n}) &= q^\frac{(r - 1 -2(i-n))(r - 1 -2(j+n))}{2} \, u_{j+n} \otimes u_{i-n} \\
&= q^\frac{(r - 1)^2}{2} q^{-(r - 1) (i+j)} q^{2(i-n)(j+n)} \, u_{j+n} \otimes u_{i-n} \\
&= q^\frac{(r - 1)^2}{2} s^{-(i+j)} q^{2(i-n)(j+n)} \, u_{j+n} \otimes u_{i-n}.
\end{align*}
The factor $q^\frac{(r - 1)^2}{2}$ is annihilated by the action of the operator $\Rs$ defined as in \eqref{eq:Bn_action_V}. Hence, using also \eqref{eq:Rcalc_3}, the action of $\Rs$ on the vector $u_i \otimes u_j$ of $\Vs_{r-1}^{\otimes 2}$ is:
\begin{align*}
\Rs(u_i &\otimes u_j) = \\
&s^{-(i+j)} \sum_{n=0}^{\min (i,r-j-1)} q^{2(i-n)(j+n)} q^{n(n-1)/2} \qbin{n+j}{j} \prod_{m=0}^{n-1} \qnum{m+j+1} \, u_{j+n} \otimes u_{i-n}.
\end{align*}

\subsection{Proof of Lemma~\ref{lem:Phi_auto}} \label{app:Phi_auto}
Let $a_{\vec{\varepsilon}} \in \AH_{n,\ell}$ as in \eqref{eq:a_epsilon} and $b \in \BH_{n,\ell}$. By the definition of the map $\Phi$ \eqref{eq:map_phi} we have that $(\Phi-1)(b) = 0$. Moreover,  for $m=1$ it holds $b_{\vec{\varepsilon}, 1} = 1$ and hence, by \eqref{eq:a_epsilon} we get $(\Phi - 1)(a_{\vec{\varepsilon}}) \in \BH_{n,\ell}$. Therefore $(\Phi - 1)^2(a_{\vec{\varepsilon}}) = 0$.

\subsection{Proof of Lemma~\ref{lem:Wnl_Anl}} \label{app:Wnl_Anl}
Let $1 \leq \ell < r$. We first show that $E \circ \Phi \vert_{\AH_{n,\ell}} = 0$. We have that:
\begin{align*}
E \circ \Phi(a_{\vec{\varepsilon}}) &= E \left( \sum_{m=0}^{\ell} b_{\vec{\varepsilon},m} \, u_0^{\otimes j - 2} \otimes u_m \otimes E^{m-1} \, u_{\vec{\varepsilon}} \right)
\\
&= \sum_{m=0}^{\ell} b_{\vec{\varepsilon},m} \, u_0^{\otimes j - 2} \otimes u_m \otimes E^m \, u_{\vec{\varepsilon}}
+ \sum_{m=0}^{\ell} b_{\vec{\varepsilon},m} \, u_0^{\otimes j - 2} \otimes u_{m-1} \otimes K E^{m-1} \, u_{\vec{\varepsilon}}.
\end{align*}
Note that $u_{\vec{\varepsilon}} \in \VH_{n-j-1, \ell - 1}$ by \eqref{eq:a_epsilon}, therefore $E^\ell u_{\vec{\varepsilon}} = 0$ and the summand for $m = \ell$ at the first sum is zero. Further, at the second sum for $m=0$ we have that $u_{m-1}=0$. Finally, $K E^{m-1} \, u_{\vec{\varepsilon}} = s^{n-j+1} q^{-2(\ell-m)} E^{m-1} \, u_{\vec{\varepsilon}}$ since $E^{m-1} u_{\vec{\varepsilon}} \in \VH_{n-j-1, \ell - m}$. Putting it all together, we have
\begin{align*}
E \circ \Phi(a_{\vec{\varepsilon}})
&= \sum_{m=0}^{\ell-1} b_{\vec{\varepsilon},m} \, u_0^{\otimes j - 2} \otimes u_m \otimes E^m \, u_{\vec{\varepsilon}}
\\
&\qquad + \sum_{m=1}^{\ell} s^{n-j+1} q^{-2(\ell-m-1)} b_{\vec{\varepsilon},m} \, u_0^{\otimes j - 2} \otimes u_{m-1} \otimes E^{m-1} \, u_{\vec{\varepsilon}}
\\
&= \sum_{m=1}^{\ell} (b_{\vec{\varepsilon},m} + s^{n-j+1} q^{-2(\ell-m-1)} b_{\vec{\varepsilon},m+1}) \, u_0^{\otimes j - 2} \otimes u_m \otimes E^m \, u_{\vec{\varepsilon}}.
\end{align*}
But $b_{\vec{\varepsilon},m} + s^{n-j+1} q^{-2(\ell-m-1)} b_{\vec{\varepsilon},m+1} = 0$ and therefore, $E \circ \Phi = 0$ on $\AH_{n,\ell}$. Now, by definition, $\Phi$ is the identity on $\BH_{n, \ell}$. Hence, $E \circ \Phi = 0 \oplus E \vert_{\BH_{n,\ell}}$.

By Lemma~\ref{lem:E_iso} $E \vert_{\BH_{n, \ell}}$ is injective for all $\ell \geq 1$ and by Lemma~\ref{lem:Phi_auto} the map $\Phi$ is an automorphism of $\VH_{n,\ell}$. Therefore,  $\ker(E \circ \Phi) \cap \VH_{n,\ell} = \ker E \cap \VH_{n,\ell} = \AH_{n,\ell}$. Hence, due to Lemma~\ref{lem:Phi_auto} we conclude that $\AH_{n,\ell} \cong \WH_{n, \ell}$.

\subsection{Proof of Equations \ref{eq:double_braiding} and \ref{eq:twist_cmr}}\label{app:double_braiding}
Let $w$ be a weight vector of $\Ps_i$ of weight $q^{(i + 2m')/2}$, where $m' \in \{-r+1,  \ldots, j+1\}$.
Since $E \, u_0^\alpha = F \, u_0^\alpha = 0$, by \eqref{eq:Rmatrix_unrolled} we have that:
\[
c_{\Vmd_0, \Ps_i} (u_0^\alpha \otimes w) = \tau \circ R (u_0^\alpha \otimes w) = \tau (q^{H \otimes H /2} u_0^\alpha \otimes w) = q^{\frac{mr (i + 2m')}{2}} \, w \otimes u_0^\alpha.
\]
Similarly, we get:
\[
c_{\Ps_i, \Vmd_0} (w \otimes u_0^\alpha) = q^{\frac{mr (i + 2m')}{2}} u_0^\alpha \otimes w.
\]
After combining the two equations, we finally have:
\[
c_{\Ps_i, \Vmd_0} \circ c_{\Vmd_0, \Ps_i} (u_0^\alpha \otimes w) = q^{mr(i+2m')} \, u_0^\alpha \otimes w = q^{mri} \, u_0^\alpha \otimes w,
\]
which proves \eqref{eq:double_braiding}.

For the calculations involving the twist operator, we use the ribbon element as given by \cite{Oh} and \cite{CGP3}, that is:
\begin{equation} \label{eq:twist_unrolled}
\theta = K^{r-1} \sum_{n=0}^{r-1} \frac{\qnum{1}^{2n}}{\qnum{n}!} q^{n(n-1)/2} S(F^n) q^{-H^2/2} E^n,
\end{equation}
where the operator $q^{-H^2/2}$ is defined as $q^{-H^2/2} v = q^{-\lambda^2 / 2} v$ for a weight vector $v$, where $\lambda$ is the strong weight of $v$ \eqref{eq:H}. Since $E \, u_0^\alpha = F \, u_0^\alpha = 0$, by \eqref{eq:twist_unrolled} we have that:
\[
\theta (u_0^\alpha) = K^{r-1} q^{-H^2/2} u_0^\alpha = q^{mr(r-1) - \frac{(mr)^2}{2}} \, u_0^\alpha.
\]
Since the action of the twist operator $\theta_{\Vs_0^\alpha}$ is defined by the action of $\theta^{-1}$, \eqref{eq:twist_cmr} is proved.

\subsection{Proof of Equation~\ref{eq:twist_Pmdi}}\label{app:twist_Pmdi}
By the naturality of the twist, it holds that:
\[
\theta_{\Pmd_i} = \theta_{\Vmd_0 \otimes \Ps_i}= \left( \theta_{\Vmd_0} \otimes \theta_{\Ps_i} \right) c_{\Ps_i, \Vmd_0} \circ c_{\Vmd_0, \Ps_i}.
\]
By \eqref{eq:double_braiding} and \eqref{eq:twist_cmr} we get:
\[
\theta_{\Vmd_0 \otimes \Ps_i}= q^{\frac{(mr)^2}{2} + mr - m r^2} (-1)^{m i } \left( \Id_{\Vmd_0} \otimes \theta_{\Ps_i} \right).
\]
Therefore, by \eqref{eq:twist_proj}:
\[
\theta_{\Pmd_i} = q^{\frac{(mr)^2+i^2}{2} + mr - m r^2 + i} (-1)^{(m+1)i} \left( \Id_{\Vmd_0} \otimes I_{1,i} - (r - i - 1) \frac{\qnum{1}^2}{\qnum{i+1}} \Id_{\Vmd_0} \otimes x_{1,i} \right).
\]
Note that $\Id_{\Vmd_0} \otimes I_{1,i} = I_{\alpha,i}$ and $\Id_{\Vmd_0} \otimes x_{1,i} = x_{\alpha,i}$. Moreover, we have that:
\[
q^{\frac{(mr)^2+i^2}{2}} = q^{\frac{(mr+i)^2}{2}} (-1)^{-mi} \andq q^{-mr^2} = (-1)^{-mr}.
\]
Combining all the above together, we get \eqref{eq:twist_Pmdi}.

\subsection{Proof of Equation~\ref{eq:Fci}}\label{app:Fci}
First we prove by induction the following:
\begin{equation} \label{eq:Fv0n}
F \, u_0^{\otimes m} = \sum_{j=1}^m s^{-(j-1)}\, c_j.
\end{equation}
For $m=1$ we have that $F \, u_0 = u_1 = c_1$. Suppose the statement holds for $n$. Then for $n+1$, we have that:
\begin{align*}
F \, u_0^{\otimes m+1} &= K^{-1} u_0 \otimes F \, u_0^{\otimes m} + F \, u_0 \otimes u_0^{\otimes m}
= s^{-1}  \sum_{j=1}^m s^{-(j-1)}\, u_0 \otimes c_j +  u_1 \otimes u_0^{\otimes m}\\
&= \sum_{j=1}^m s^{-j}\, c_{j+1} + (-1)^k \, c_1
= \sum_{j=1}^{m+1} s^{-(j-1)}\, c_{j}.
\end{align*}
Now by \eqref{eq:Fv0n} we have that: 
\begin{align*}
F \, (u_1 \otimes & u_0^{\otimes m}) = K^{-1} u_1 \otimes F \, u_0^{\otimes m} + F \, u_1 \otimes u_0^{\otimes m}
\\
&= s^{-1} q^2 \sum_{j=1}^m s^{-(j-1)}\, u_1 \otimes c_j + \qint{2}^2 \, u_2 \otimes u_0^{\otimes m}
= q^2 \sum_{j=1}^m s^{-j}\, a_{1,j+1} + \qint{2}^2 \, b_1 .
\end{align*} 
Note that $\qint{2} \neq 0$, since $2+\ell<r$. And finally, we have that:
\begin{align*}
&F \, c_i = F \,(u_0^{\otimes i-1} \otimes u_1 \otimes u_0^{\otimes n-i})
= K^{-1} u_0^{\otimes i-1} \otimes F \, (u_1 \otimes u_0^{\otimes n-i}) + F \, u_0^{\otimes i-1} \otimes u_1 \otimes u_0^{\otimes n-i}
\\
&= s^{-(i-1)} \left[ q^2 \sum_{j=1}^{n-i} s^{-j} \, u_0^{\otimes i-1} \otimes a_{1,j+1} + \qint{2}^2 u_0^{\otimes i -1} \otimes b_1 \right]
+ \sum_{j=1}^{i-1} s^{-(j-1)}\, c_j \otimes u_1 \otimes u_0^{\otimes n-i}
\\
&=  s^{-(i-1)} q^2 \sum_{j=1}^{n-i} s^{-j} \, a_{i,i+j} 
+  \sum_{j=1}^{i-1} s^{-(j-1)}\, a_{j,i} + \qint{2}^2 s^{-(i-1)} \; b_i.
\end{align*} 

\subsection{Proof of Lemma~\ref{lem:min_pol}} \label{app:min_pol}
We prove the statement for the matrix corresponding to $\sigma_1 \in B_n$. Then it follows immediately for the rest of the generators of $B_n$ since the generators of $B_n$ are all conjugate to each other. Since $p(X)$ is the minimal polynomial of the representation $\WH_{n,2}$, we have that
$p(\sigma_1) w_{i,j} = 0$, for all $1 \leq i < j \leq n$. It remains to prove the same for the additional basis vectors of $\NH_{n,2,0}$ and $\NH_{n,2,1}$. We also have that:
\[
p(X) = X^3 + (-1 + s^{-2} - s^{-4} q^2) X^2 + (-s^{-2} + s^{-4} q^2 - s^{-6} q^2) X + s^{-6} q^2.
\]

We start with $\NH_{n,2,0}$. We denote $t = s^{-3} (1-q^2)$. By an induction on $k$ it  holds that:
\[
\sigma_1^k b = b + t \sum_{m=0}^{k-1} (s^{-4}q^2)^m \, w_{1,2}.
\]
Therefore, the coefficient of $b$ in $p(\sigma_1) b$ is zero, since the sum of the coefficients of $p(X)$ is zero. Moreover, the coefficient of $w_{1,2}$ in $p(\sigma_1) b$ equals:
\begin{align*}
&t(1 + s^{-4}q^2 + s^{-8}q^4)  + t(-1 + s^{-2} - s^{-4} q^2) (1 + s^{-4}q^2) + t(-s^{-2} + s^{-4} q^2 - s^{-6} q^4)
\\
&= 0.
\end{align*}
Hence, the matrix for $\sigma_1$ satisfies $p(X)$ and $p(X)$ is the minimal polynomial for the matrix (since it is the minimal polynomial for the matrices of $\WH_{n,2}$ when the eigenvalues are distinct, that is when $r \geq 5$).

We proceed now with $\NH_{n,2,1}$. Since $\sigma_1 b'_j = b_j'$, for $3 \leq j < n$, and since the coefficients of $p(X)$ sum up to zero, we have that $p(\sigma_1)b_j' = 0$, for $3 \leq j < n$. For $j=1,2$ we solve the following equation:
\[
\sigma_1^3 b'_j + x \, \sigma_1^2 b'_j + y \, \sigma_1 b'_j + z \, b'_j = 0.
\]
Note that:
\begin{align*}
\sigma_1^3 b'_2 + x \, \sigma_1^2 b'_2 + y \, \sigma_1 b'_2 + z \, b'_2 = 0
&\Leftrightarrow
s^{-1} \sigma_1^2 b'_1 + s^{-1} x \, \sigma_1 b'_1 + s^{-1} y \, \sigma_1 b'_1 + z \, b'_2 = 0
\\
&\Leftrightarrow
s^{-1} \sigma_1^3 b'_1 + s^{-1} x \, \sigma_1^2 b'_1 + s^{-1} y \, \sigma_1^2 b'_1 + s^{-1}  z \, b'_1 = 0
\\
&\Leftrightarrow
\sigma_1^3 b'_1 + x \, \sigma_1^2 b'_1 + y \, \sigma_1^2 b'_1 + z \,  b'_1 = 0.
\end{align*}
So, it suffices to solve the equation for either $b'_1$ or $b'_2$. Calculating $\sigma_1^k b'_2$, for $k=1,2,3$ we find that the equation is satisfied when
$x=-1 + s^{-2} - s^{-4} q^2$, $y=-s^{-2} + s^{-4} q^2 - s^{-6} q^2$ and $z=s^{-6} q^2$. We provide a verification with a Mathematica program \texttt{Nn2\_min\_poly.nb} (available at \cite{prog}).
So, $p(\sigma_1)b_j' = 0$, for every $1 \leq j <n$ and similarly as before, $p(X)$ is the minimal polynomial for the matrix corresponding to $\sigma_1$.

\subsection{Proof of Prop.~\ref{prop:Nn20}}\label{app:Nn20}
Note that the action of $B_n$ on the submodule spanned by the vectors $w_{i,j}$ is the same as the action of $B_n$ on the vectors $w_{i,j}$ of $\WH_{n,2}$ as in \eqref{eq:LKB}, which is isomorphic to the LKB representation. Therefore, it remains to prove that the braid group relations are satisfied on the vector $b$. Let $1 \leq i, j \neq n-1$ such that $\lvert i -j \rvert > 1$. Then:
\begin{align*}
\sigma_j \sigma_i b
&= t \, \sigma_j w_{i,i+1} + \sigma_j b
= t \, \sigma_j w_{i,i+1} + t \, w_{j,j+1} + b
\\
&= t \, w_{i,i+1} + t \, \sigma_i w_{j,j+1} + b
= t \, \sigma_i w_{j,j+1} + \sigma_i b
= \sigma_i \sigma_j b,
\end{align*}
where at third equality we have used the fact that $\lvert i -j \rvert > 1$. Moreover, for $1 \leq i < n-2$, by applying the formula for the action of $B_n$ on $b$, we have that:
\begin{align*}
\sigma_i \sigma_{i+1} \sigma_i b &= t \left( \sigma_i \sigma_{i+1} w_{i,i+1} + \sigma_i w_{i+1,i+2} + w_{i,i+1} \right) + b,
\\
\sigma_{i+1} \sigma_i \sigma_{i+1} b &= t \left( \sigma_{i+1} \sigma_i w_{i+1,i+2} + \sigma_{i+1} w_{i,i+1} + w_{i+1,i+2} \right) + b.
\end{align*}
A simple calculation using the action of $B_n$ on the vectors $w_{i,j}$ shows that the two expressions in parentheses are both equal to $s^{-1} \, w_{i,i+2} + w_{i,i+1} + s^{-2} \, w_{i+1,i+2}$, which proves the existence of the $B_n$-representation $\widetilde{\NH}_{n,2,0}$ over $\Z[q^{\pm}, s^{\pm}, t^{\pm}]$. Furthermore, the variable $t$ does not appear in any negative powers since the action of the inverses of the braid group generators on the vector $b$ is given by:
\[
\sigma_i^{-1} = - s^4 q^{-2} t \, w_{i,i+1} + b.
\]
That is, the representation is actually defined over $\Z[q^{\pm}, s^{\pm}, t]$, which concludes the proof of the statement.

\subsection{Proof of Proposition~\ref{prop:Delta_n20}} \label{app:Delta_n20}
We prove first the statement for $k=1$. We first prove that for every $m \leq n$:
\begin{equation} \label{eq:sigma_m_b}
(\sigma_1 \ldots \sigma_{m-1})(\sigma_1 \ldots \sigma_{m-2}) \ldots \sigma_1 b = b + t \sum_{1 \leq i, j \leq m} w_{i,j}
\end{equation}
For $m=2$, we have that $\sigma_1 \, b = b + t \, w_{1,2}$. Suppose the statement holds for any number less than $m \leq n$. Then for $m-1$ it holds that:
\[
(\sigma_1 \ldots \sigma_{m-2}) \ldots \sigma_1 b = b + t \sum_{1 \leq i, j \leq m-1} w_{i,j}.
\]
Now, for $m$ we have that
\[
(\sigma_1 \ldots \sigma_{m-1}) (\sigma_1 \ldots \sigma_{m-2}) \ldots \sigma_1 b
= (\sigma_1 \ldots \sigma_{m-1}) b + t \sum_{1 \leq i, j \leq m-1} (\sigma_1 \ldots \sigma_{m-1}) w_{i,j}.
\]
Using \eqref{eq:LKB} with $s=q=1$ we get that:
\begin{align*}
(\sigma_1 \ldots \sigma_{m-1}) w_{i,j}
&= \sigma_1 \ldots \sigma_j w_{i,j} 
= \sigma_1 \ldots \sigma_{j-1} w_{i,j+1} 
= \sigma_1 \ldots \sigma_i w_{i,j+1}
\\
&= \sigma_1 \ldots \sigma_{i-1} w_{i+1,j+1} = w_{i+1,j+1}.
\end{align*}
Now, we prove that for every $m'<n$:
\[
\sigma_1 \ldots \sigma_m' b = b + t \sum_{2 \leq j \leq m'+1} w_{1,j}.
\]
For $m'=2$ the statement is obvious. Supposing the statement for any number less than $m' < n$, we have that:
\begin{align*}
\sigma_1 \ldots \sigma_{m'} b &= \sigma_1 \ldots \sigma_{m'-1} b + t  \, \sigma_1 \ldots \sigma_{m'} w_{m',m'+1}
= b + t \sum_{2 \leq j \leq m'} w_{1,j} + t \, w_{1,m'+1}
\\
&= b + t \sum_{2 \leq j \leq m'+1} w_{1,j},
\end{align*}
where at the second equality we use the inductive statement and \eqref{eq:LKB}. Combining all the above, it holds that:
\begin{align*}
(\sigma_1 \ldots \sigma_{m-1}) (\sigma_1 \ldots \sigma_{m-2}) \ldots \sigma_1 b
&= b + t \sum_{2 \leq j \leq m} w_{1,j} + t \sum_{1 \leq i, j \leq m-1} w_{i+1,j+1}
\\
&= b + t \sum_{2 \leq j \leq m} w_{1,j} + t \sum_{2 \leq i, j \leq m} w_{i,j},
\end{align*}
which proves \eqref{eq:sigma_m_b}. By substituting $m=n$ in \eqref{eq:sigma_m_b} we get \eqref{eq:Delta_sq1} for $k=1$.

We now use induction on $k$ to prove \eqref{eq:Delta_sq1} for any $k \in \Np$. Suppose that the statement holds for any number less than $k$. Then:
\begin{align*}
\Delta_n^k b
= \Delta_n \Delta_n^{k-1} b
= \Delta_n b + (k-1) t \, \Delta_n \sum_{1 \leq i < j \leq n} w_{i,j}.
\end{align*}
As mentioned before, the matrices corresponding to the generators of $B_n$ for the representation $\overline{\WH}_{n,2}$ are permutation matrices. Therefore, any braid in $B_n$ acts trivially on $\sum_{1 \leq i, j \leq m} w_{i,j}$. Therefore:
\[
\Delta_n^k b = \Delta_n b + (k-1) t \sum_{1 \leq i < j \leq n} w_{i,j}
= b + k t \sum_{1 \leq i < j \leq n} w_{i,j},
\]
and the statement is proved for every $k \in \Np$. Now, it is easy to see that
\[
\Delta_n^{-1} = b - kt \sum_{1 \leq i, j \leq n} w_{i,j}.
\]
By a similar inductive argument, we can prove the statement for every $k < 0$, which concludes the proof.

\subsection{Proof of Prop.~\ref{prop:Nn21}}\label{app:Nn21}
It remains to show that the braid group relations are satisfied for the vectors $b'_m$, with $1 \leq m \leq n-1$. We start with the commutation relations $\sigma_i \sigma_j = \sigma_j \sigma_i$ for $\vert i - j \vert > 1$. Suppose first that $i,j \neq n-1$ and that $m \notin \{ i, i+1, j, j+1 \}$. Then we have that $\sigma_i \sigma_j b'_m = b'_m = \sigma_j \sigma_m b'_m$. If $m=i$ (or without loss of generality $m=j$) then:
\[
\sigma_i \sigma_j b'_i = t \, w_{i,i+1} + (1-s^{-2}) \, b'_i + s^{-1} \, b'_{i+1} = \sigma_j \sigma_i b'_i.
\]
Similarly if $m=i+1$ (or $m=j+1$):
\[
\sigma_i \sigma_j b'_{i+1} = s^{-1} \, b'_i = \sigma_j \sigma_i b'_{i+1}.
\]
Now, suppose $i$ or $j$ is equal to $n-1$ (we choose without loss of generality $i=n-1$ and hence $j < n-2$) and $1 \leq m < n-1$. If $m \neq j, j+1$, we have that:
\[
\sigma_{n-1} \sigma_j b'_m = b'_m - s^{n-k-1} \, b'_{n-1} = \sigma_j \sigma_{n-1} b'_m.
\]
For $m=j$:
\[
\sigma_{n-1} \sigma_j b'_j =  s^{j-n} t \, w_{j,j+1} + (1-s^{-2}) \, b'_j  + s^{-1} \, b'_{j+1} - s^{n-j-1} \, b'_{n-1} = \sigma_j \sigma_{n-1} b'_j.
\]
For $m=j+1$:
\[
\sigma_{n-1} \sigma_j b'_{j+1} = s^{-1} \, b '_j - s^{-1} s^{n-j-1} \, b'_{n-1} = \sigma_j \sigma_{n-1} b'_{j+1}.
\]
Finally, if $m=n-1$:
\[
\sigma_{n-1} \sigma_j b'_{n-1}  = s^{-1} t \, w_{n-1,n} - s^{-2} \, b'_{n-1} = \sigma_j \sigma_{n-1} b'_{n-1}.
\]

Now we proceed to the braiding relations $\sigma_i \sigma_{i+1} \sigma_i = \sigma_{i+1} \sigma_i \sigma_{i+1}$, for $1 \leq i < n-1$. We start with the case $1 \leq i < n-2$. If $m \neq i, i+1, i+2$ then the braiding relations are satisfied since $\sigma_i b'_m = b'_m = \sigma_{i+1} b'_k$. For $k=i$ we have that:
\begin{align*}
\sigma_i \sigma_{i+1} \sigma_i b'_i = s^{i-n} t & \overbrace{\left( \sigma_i \sigma_{i+1} w_{i,i+1} + \sigma_i w_{i+1, i+2} + (1 - s^{-2}) \, w_{i,i+1} \right)}^{A_1 :=}
\\
&+ (1- s^{-2}) \, b'_i + s^{-1} (1- s^{-2}) \, b'_{i+1} + s^{-2} \, b'_{i+2}.
\end{align*}
On the other hand:
\begin{align*}
\sigma_{i+1} \sigma_i \sigma_{i+1} b'_i = s^{i-n} t & \overbrace{\left( \sigma_{i+1} w_{i,i+1} + w_{i+1, i+2} \right)}^{A_2 :=}
\\
&+ (1- s^{-2}) \, b'_i + s^{-1} (1- s^{-2}) \, b'_{i+1} + s^{-2} \, b'_{i+2}.
\end{align*}
Note that in both expressions we have the same terms involving the vectors $b'_k$, for $k=1, \ldots, n$. It remains to compare the expressions $A_1$ and $A_2$. In Appendix~\ref{app:Nn20} it was proved that:
\begin{equation} \label{eq:braid_rel_wij}
\sigma_i \sigma_{i+1} w_{i,i+1} + \sigma_i w_{i+1,i+2} + w_{i,i+1} = \sigma_{i+1} \sigma_i w_{i+1,i+2} + \sigma_{i+1} w_{i,i+1} + w_{i+1,i+2}.
\end{equation}
Therefore:
\[
A_1 - A_2 = \sigma_{i+1} \sigma_i w_{i+1, i+2} - s^{-2} \, w_{i,i+1} = 0,
\]
where the last equality is due to \eqref{eq:LKB}. Hence, $\sigma_i \sigma_{i+1} \sigma_i b'_i = \sigma_{i+1} \sigma_i \sigma_{i+1} b'_i$.

\noindent
For $m=i+1$, it holds that:
\[
\sigma_i \sigma_{i+1} \sigma_i b'_{i+1} = s^{i-n-1} t \, w_{i+1,i+2} + s^{-1} (1-s^{-2}) \, b'i + s^{-2} \, b'_{i+1} =  \sigma_{i+1} \sigma_i \sigma_{i+1} b'_{i+1}
\]
If $m=i+2$, then:
\[
\sigma_i \sigma_{i+1} \sigma_i b'_{i+2} = s^{-2} \, b'_i = \sigma_{i+1} \sigma_i \sigma_{i+1} b'_{i+2}.
\]
Now we examine the case $i=n-2$. If $m<n-2$, then we have that:
\begin{align*}
\sigma_{n-2} \sigma_{n-1} \sigma_{n-2} b'_{n-2} = s^{-2} t & \overbrace{\left( \sigma_{n-2} \sigma_{n-1} w_{n-2,n-1} + \sigma_{n-2} w_{n-1, n} + (1 - s^{-2}) \, w_{n-2,n-1} \right)}^{B_1 :=}
\\
&- s^{-2} \, b'_{n-2} + s^{-1} (1- s^{-2}) \, b'_{n-1}.
\end{align*}
On the other hand:
\[
\sigma_{n-1} \sigma_{n-2} \sigma_{n-1} b'_{n-2} = s^{-2} t \overbrace{\left( \sigma_{n-1} w_{n-2,n-1} + w_{n-1, n} \right)}^{B_2 :=}
- s^{-2} \, b'_{n-2} + s^{-1} (1- s^{-2}) \, b'_{n-1}.
\]
Again, due to \eqref{eq:braid_rel_wij} we have that $B_1 = B_2$ (the proof is analogous to the proof of $A_1 = A_2$). Hence, $\sigma_{n-2} \sigma_{n-1} \sigma_{n-2} b'_{n-2} = \sigma_{n-1} \sigma_{n-2} \sigma_{n-1} b'_{n-2}$.

\noindent
And, finally, if $m=n-1$:
\[
\sigma_{n-2}  \sigma_{n-1} \sigma_{n-2} b'_{n-1}
= s^{-3} t \, w_{n-2,n-1} - s^{-3} \, b'_{n-2} + s^{-2} \, b'_{n-1}
= \sigma_{n-1} \sigma_{n-2} \sigma_{n-1} b'_{n-1},
\]
which completes the proof.

\subsection{Proof of Equation~\ref{eq:phi_Burau}} \label{app:phi_Burau}
Using the fact that $\beta_i=1$ for $n \equiv -2 \mod r$ (see Section~\ref{sec:Nn2}) and \eqref{eq:Fci2}, we have that:
\[
F\, c_i = s^{-(i-1)} q^2 \sum_{j=1}^{n-i} s^{-j} \, w_{i,i+j} + \sum_{j=1}^{i-1} s^{-(j-1)} \, w_{j,i} + s^{-(i-1)} \sum_{j=1}^{n} b_{j}.
\]
So:
\begin{align*}
F \, \overline{c}_i = & F \, c_i - s^{n-i} F \, c_n =
 s^{-(i-1)} q^2 \sum_{j=1}^{n-i} s^{-j} \, w_{i,i+j} + \sum_{j=1}^{i-1} s^{-(j-1)} \, w_{j,i} + s^{-(i-1)} \sum_{j=1}^{n} b_{j}
\\
&\quad - s^{n-i} \sum_{j=1}^{n-1} s^{-(j-1)} \, w_{j,n} - s^{i-1} \sum_{j=1}^n b_j
\\
&= s^{-(i-1)} q^2 \sum_{j=1}^{n-i} s^{-j} \, w_{i,i+j} + \sum_{j=1}^{i-1} s^{-(j-1)} \, w_{j,i} - s^{n-i} \sum_{j=1}^{n-1} s^{-(j-1)} \, w_{j,n}.
\end{align*}


\begin{thebibliography}{10}
\bibitem{BBG}
A.~Beliakova, C.~Blanchet, and N.~Geer.
{\em Logarithmic Hennings invariants for restricted quantum sl(2)}.
Algebr. Geom. Topol. 18 (2018), no. 7, pp. 4329--4358. 

\bibitem{BCGP1}
C.~Blanchet, F.~Costantino, N.~Geer, and B.~Patureau{-}Mirand.
{\em Non-semi-simple TQFTs, Reidemeister torsion and Kashaev's invariants}.
Advances in Mathematics 301 (2016), pp. 1 -- 78.

\bibitem{Bi1}
S.~Bigelow.
{\em Braid groups are linear}. J. Amer. Math. Soc. 14.2 (2001), 471--486.

\bibitem{Bir}
J.~S.~Birman.
{\em Braids, links, and mapping class groups}.
Annals of Mathematics Studies, No. 82. Princeton University Press, University of Tokyo Press, Tokyo, 1974, pp. ix+228.

\bibitem{BKL}
J.~S.Birman, K.~H.~Ko, and S.~J.~Lee.
{\em A new approach to the word and conjugacy problems in the braid groups}.
Adv. Math. 139.2 (1998), pp. 322--353.

\bibitem{BW}
J.~S.~Birman and H.~Wenzl.
{\em Braids, link polynomials and a new algebra}.
Trans. Amer. Math. Soc. 313.1 (1989), pp. 249--273.

\bibitem{BM}
M.~Brou\'e and G.~Malle.
{\em Zyklotomische Heckealgebren}.
Ast\'erisque, 212 (1993). Repr\'esentations unipotentes g\'en\'eriques et blocs des groupes r\'eductifs finis, pp. 119--189.

\bibitem{Bur}
W.~Burau.
{\em \"Uber Zopfgruppen und gleichsinnig verdrillte Verkettungen}.
Abh. Math. Sem. Univ. Hamburg 11.1 (1935), pp. 179--186.


\bibitem{CGP1}
F.~Costantino, N.~Geer, and B.~Patureau{-}Mirand.
{\em Quantum invariants of 3-manifolds via link surgery presentations and non-semi-simple categories}.
J. Topol. 7.4 (2014), pp. 1005--1053.

\bibitem{CGP3}
F.~Costantino, N.~Geer, and B.~Patureau{-}Mirand.
{\em Some remarks on the unrolled quantum group of {$\germ{sl}(2)$}}.
J. Pure Appl. Algebra 219.8 (2015), pp. 3238--3262. (See also the revised version at arXiv:1406.0410v3 [math.GT] for corrected formulas).

\bibitem{FGST}
B.~L.~Feigin, A.~M.~Gainutdinov, A.~M.~Semikhatov, and I.~Yu.~Tipunin.
{\em Modular group representations and fusion in logarithmic conformal field theories and in the quantum group center}.
Comm. Math. Phys. 265.1 (2006), pp. 47--93.

\bibitem{Gar}
F.~A.~Garside.
{\em The braid group and other groups}.
Quart. J. Math. Oxford Ser. (2) 20 (1969), pp. 235--254.

\bibitem{Ito1}
T.~Ito.
{\em Topological formula of the loop expansion of the colored Jones polynomials}, 2014. arXiv: 1411.5418 [math.GT].

\bibitem{Ito2}
T.~Ito.
{\em Reading the dual Garside length of braids from homological and quantum representations}.
Comm. Math. Phys. 335.1 (2015), pp. 345--367.

\bibitem{Ito3}
T.~Ito.
{\em A homological representation formula of colored Alexander invariants}.
Adv. Math. 289 (2016), pp. 142--160.

\bibitem{JK}
C.~Jackson and T.~Kerler.
{\em The Lawrence-Krammer-Bigelow representations of the braid groups via {$U_q(\germ{sl}_2)$}}.
Adv. Math. 228.3 (2011), pp. 1689--1717.

\bibitem{KarPhd}
K.~Karvounis.
{\em Braid group action on projective modules of quantum $\mathfrak{sl}(2)$}.
PhD thesis, Universit\"at Z\"urich, 2019.

\bibitem{prog}
K.~Karvounis. Mathematica programs. \url{https://github.com/karvounisk/BraidReps}.

\bibitem{KaTu}
C.~Kassel and V.~Turaev.
{\em Braid groups}.
Vol. 247. Graduate Texts in Mathematics.
With the graphical assistance of Olivier Dodane.
Springer, New York, 2008, pp. xii+340.


\bibitem{KonSa}
H.~Kondo and Y.~Saito.
{\em Indecomposable decomposition of tensor products of modules over the restricted quantum universal enveloping algebra associated to {${\germ{sl}}_2$}}.
J. Algebra 330 (2001), pp. 103--129.

\bibitem{Koh}
T.~Kohno.
{\em Quantum and homological representations of braid groups}.
Configuration spaces Vol.~14. CRM Series. Ed. Norm., Pisa, 2012, pp. 355--372.

\bibitem{Kr1}
D.~Krammer.
{\em The braid group {$B_4$} is linear}.
Invent. Math. 142.3 (2000), pp. 451--486.

\bibitem{Kr2}
D.~Krammer.
{\em Braid groups are linear}.
Ann. of Math. (2) 155.1 (2002), pp. 131--156.

\bibitem{Law}
R.~J.~Lawrence.
{\em Homological representations of the Hecke algebra}.
Comm. Math. Phys. 135.1 (1990), pp. 141--191.

\bibitem{MarThese}
I.~Marin.
{\em Repr\'esentations linéaires des tresses infinit\'esimales}.
PhD thesis, Universit\'e Paris 11 - Orsay, 2001.

\bibitem{Mar}
I.~Marin.
{\em The cubic Hecke algebra on at most 5 strands}.
J. Pure Appl. Algebra 216.12 (2012), pp. 2754--2782.

\bibitem{MW1}
I.~Marin and E.~Wagner.
{\em A cubic defining algebra for the Links-Gould polynomial}.
Adv. Math. 248 (2013), pp. 1332--1365.

\bibitem{MW2}
I.~Marin and E.~Wagner.
{\em Markov traces on the Birman-Wenzl-Murakami algebras} (2014). arXiv: 1403.4021 [math.GT].

\bibitem{MuBMW}
J.~Murakami.
{\em The Kauffman polynomial of links and representation theory}.
Osaka J. Math. 24.4 (1987), pp. 745--758.

\bibitem{Oh}
T.~Ohtsuki.
{\em Quantum invariants}, Vol.~29. Series on Knots and Everything.
World Scientific Publishing Co., Inc., River Edge, NJ, 2002, pp. xiv+489.

\bibitem{Sta}
R.~P.~Stanley.
{\em Enumerative combinatorics. Volume 1}.
Vol.~49. Cambridge Studies in Advanced Mathematics.
Cambridge University Press, Cambridge, second edition, 2012, pp. xiv+489.

\bibitem{X2}
J.~Xiao.
{\em Generic modules over the quantum group {$U_t({\rm sl}(2))$} at {$t$} a root of unity}.
Manuscripta Math. 83.1 (1994), pp. 75--98.

\bibitem{X3}
J.~Xiao.
{\em Finite-dimensional representations of {$U_t({\rm sl}(2))$} at roots of unity}.
Canad. J. Math., 49.4 (1997), pp. 772--787.

\bibitem{Zi}
M.~G.~Zinno.
{\em On Krammer's representation of the braid group}.
Math. Ann. 321.1 (2001), pp. 197--211.

\end{thebibliography}
\end{document}